\documentclass[11pt,reqno]{amsart}
\usepackage{amssymb}
\usepackage{amsmath}
\usepackage{amsthm}
\usepackage{graphicx}
\usepackage{caption}
\usepackage{bm}
\usepackage[colorlinks=true,urlcolor=blue,
citecolor=red,linkcolor=blue,linktocpage,pdfpagelabels,
bookmarksnumbered,bookmarksopen]{hyperref}
\usepackage[titletoc,title]{appendix}
\numberwithin{equation}{section} 
\newcounter{cont}[section] 

\newtheorem{thm}[cont]{Theorem}
\newtheorem{prop}[cont]{Proposition}
\newtheorem{lem}[cont]{Lemma}

\theoremstyle{definition}

 \theoremstyle{remark}
 \newtheorem{rem}[cont]{Remark}

\newcommand{\R}{\mathbb{R}}
\newcommand{\e}{\varepsilon}

\begin{document}
\title[Hyperbolic mass conserving Allen--Cahn eq. in 1D]{Metastable dynamics for a hyperbolic variant \\of the mass conserving Allen--Cahn equation\\ in one space dimension}

\author[R. Folino]{Raffaele Folino}
\address[Raffaele Folino]{Departamento de Matem\'aticas y Mec\'anica, IIMAS - UNAM (Mexico)}
\email{folino@mym.iimas.unam.mx}


\maketitle

\begin{abstract}
In this paper, we consider some hyperbolic variants of the mass conserving Allen--Cahn equation,
which is a nonlocal reaction-diffusion equation, introduced (as a simpler alternative to the Cahn--Hilliard equation) to describe phase separation in binary mixtures.
In particular, we focus our attention on the metastable dynamics of some solutions to the equation 
in a bounded interval of the real line with homogeneous Neumann boundary conditions.
It is shown that the evolution of profiles with $N+1$ transition layers is very slow and we derive a system of ODEs, 
which describes the exponentially slow motion of the layers.
A comparison with the classical Allen--Cahn and Cahn--Hilliard equations and theirs hyperbolic variations is also performed. 
\end{abstract}

\section{Introduction}
The goal of this paper is to study the metastable dynamics of the solutions to the \emph{hyperbolic mass-conserving Allen--Cahn equation}
\begin{equation}\label{eq:hyp-nonlocal}
	\tau u_{tt}+g(u)u_t+\int_0^1\left[1-g(u)\right]u_t\,dx=\e^2 u_{xx}+f(u)-\int_0^1f(u)\,dx,
\end{equation}
where $u=u(x,t): (0,1)\times(0,+\infty)\rightarrow\mathbb R$,
subject to homogeneous Neumann boundary conditions
\begin{equation}\label{eq:Neumann}
	u_x(0,t)=u_x(1,t)=0, \qquad \qquad t>0,
\end{equation}
and initial data
\begin{equation}\label{eq:initial}
	u(x,0)=u_0(x), \qquad u_t(x,0)=u_1(x), \qquad \qquad x\in[0,1].
\end{equation}
Precisely, we are interested in the behavior of the solutions to the initial boundary value problem \eqref{eq:hyp-nonlocal}-\eqref{eq:Neumann}-\eqref{eq:initial}, 
when the diffusion coefficient $\e^2$ is very small (and strictly positive), the initial data $u_0,u_1$ satisfy appropriate assumptions that will be specified later, 
the damping coefficient $g\in C^1(\R)$ is strictly positive, namely
\begin{equation}\label{eq:ass-g}
	g(u)\geq\sigma>0,\qquad \forall\,u\in\R,
\end{equation} 
and $f:\R\to\R$ is a balanced bistable reaction term, that is we assume $f=-F'$, where $F\in C^3(\mathbb{R})$ satisfies 
\begin{equation}\label{eq:ass-F}
	F(\pm1)=F'(\pm1)=0, \qquad F''(\pm1)>0, \qquad F(u)>0, \; \, \forall\,u\neq\pm1.
\end{equation}
In other words, $-f$ is the derivative of a double well potential with wells of equal depth located at $\pm1$;
the typical example is $F(u)=\frac14(u^2-1)^2$.

Formally, by taking $\tau=0$ and $g\equiv1$ in \eqref{eq:hyp-nonlocal}, one obtains the celebrated \emph{mass conserving Allen--Cahn equation} in one space dimension
\begin{equation}\label{eq:maco-AC}
	u_t=\e^2 u_{xx}+f(u)-\int_0^1f(u)\,dx.
\end{equation}
Before presenting our results, we do a short historical review on the mass conserving Allen--Cahn equation \eqref{eq:maco-AC}
and we show how to formally derive the hyperbolic variant \eqref{eq:hyp-nonlocal}. 

\subsection{Mass conserving Allen--Cahn equation}
In \cite{RubSte}, Rubinstein and Sternberg introduced the following \emph{nonlocal reaction-diffusion equation}
\begin{equation}\label{eq:nonlocal-AC-multiD}
	u_t=\Delta u+f(u)-\lambda_f, \qquad \bm x\in\Omega, \, t>0,
\end{equation}
with no-flux boundary conditions
\begin{equation*}
 	\bm n \cdot \nabla u=0,  \qquad \bm x\in\partial\Omega,
\end{equation*}
where $u=u(\bm x,t): \Omega\times(0,+\infty)\rightarrow\mathbb R$, $\Omega\subset\mathbb R^n$ is a smooth bounded domain 
with outer unit normal $\bm n$ and total volume $|\Omega|$,
the reaction term $f$ is equal to $-F'$, where $F$ is a double well potential, and 
\begin{equation*}
	\lambda_f:=\frac{1}{|\Omega|}\int_\Omega f(u)\,dx.
\end{equation*}
Rubinstein and Sternberg proposed equation \eqref{eq:nonlocal-AC-multiD} to model phase separation after rapid cooling of homogeneous binary systems 
(such as glasses and polymers).
If we omit the term $\lambda_f$ in \eqref{eq:nonlocal-AC-multiD}, we obtain a (parabolic) reaction-diffusion equation and when $f=-F'$
with $F$ satisfying \eqref{eq:ass-F}, we have the bistable equation known as \emph{Allen--Cahn equation}
\begin{equation}\label{eq:AC-multiD}
	u_t=\Delta u+f(u),
\end{equation} 
which has been originally proposed in \cite{Allen-Cahn} to describe the motion of antiphase boundaries in iron alloys.
The presence of the term $\lambda_f$ implies the conservation of the mass of the solutions:
by integrating equation \eqref{eq:nonlocal-AC-multiD} in $\Omega$ and using the no-flux boundary conditions we infer
\begin{equation*}
	m(t):=\int_\Omega u(\bm x,t)\,dx=\int_\Omega u(\bm x,0)\,dx, \qquad \qquad \forall\,t\geq0.
\end{equation*}
Therefore, equation \eqref{eq:nonlocal-AC-multiD} is a reaction-diffusion equation with the important property that 
the total mass is preserved in time and it was proposed as a simpler alternative to the \emph{Cahn--Hilliard equation} \cite{Cahn-Hill}
\begin{equation}\label{eq:CH-multiD}
	u_t=-\Delta\left(\Delta u+f(u)\right).
\end{equation}
Let us briefly compare the mass conserving Allen--Cahn equation \eqref{eq:nonlocal-AC-multiD} 
with respect to the Allen--Cahn \eqref{eq:AC-multiD} and Cahn--Hiliard \eqref{eq:CH-multiD} equations (for details see \cite{Bron-Stoth,MurRin,RubSte}).
As \eqref{eq:AC-multiD}, equation \eqref{eq:nonlocal-AC-multiD} is a second order PDE and it can be seen as the gradient flow in $L^2$ for the functional
\begin{equation*}
	E[u]:=\int_\Omega\left[\frac12|\nabla u|^2+F(u)\right]\,dx.
\end{equation*}
More precisely, the solutions of equations \eqref{eq:nonlocal-AC-multiD}-\eqref{eq:AC-multiD} with no-flux boundary conditions satisfy
\begin{equation*}
	\frac{d}{dt}E[u](t)=-\int_\Omega u_t^2(\bm x,t)\,dx.
\end{equation*}
On the contrary, in the case of \eqref{eq:nonlocal-AC-multiD} we have conservation of mass and the stationary solutions are the same of \eqref{eq:CH-multiD}.
In particular, notice that the only constant equilibria for \eqref{eq:AC-multiD} are the zeros of $f$, 
while all the constants $c\in\R$ are equilibria for \eqref{eq:nonlocal-AC-multiD} and \eqref{eq:CH-multiD}.

Nonetheless, the behavior of the solutions to the three equations \eqref{eq:nonlocal-AC-multiD}-\eqref{eq:AC-multiD}-\eqref{eq:CH-multiD} is rather different.
It is impossible to mention all the results, but we briefly recall that the solutions of the one-dimensional Allen--Cahn equation exhibit the phenomenon of \emph{metastability} 
and we have persistence of unstable structure for an exponentially long time \cite{Bron-Kohn,Carr-Pego,Carr-Pego2,Chen,Fusco-Hale}, 
while in the multidimensional case, equation \eqref{eq:AC-multiD} is strictly related to the motion by mean curvature flow \cite{Bron-Kohn2,Chen2,deM-Sch}.
Roughly speaking, if we add a small diffusion coefficient $\e^2$ in \eqref{eq:AC-multiD} and consider an initial datum with 
finitely many sign changes in $\Omega$, then in a first phase, the solution $u$ behaves as if there were no diffusion and develops steep interfaces;
after that, diffusion plays a crucial role and it is very interesting to study the propagation of the \emph{interface} $\Gamma_t:=\left\{\bm x\in\Omega : u(\bm x,t)=0\right\}$.
In the one-dimensional case, $\Gamma_t$ consists of a finite number of points and they move with an exponentially small velocity $\mathcal{O}(\exp(-C/\e))$ as $\e\to0^+$;
in the multi-dimensional case the interface moves by mean curvature flow and its velocity is of order $\e^2$.

It is very interesting to study the propagation of the interface also when the mass is conserved:
for the one-dimensional case, we recall the contributions \cite{ReyWar,SunWard} and \cite{Bates-Xun1,Bates-Xun2}, 
where the authors study the metastable dynamics of the solutions for the mass conserving Allen--Cahn and the Cahn--Hilliard equations, respectively. 
In the multi-dimensional case, we mention \cite{Bron-Stoth,ChHiLo,MurRin} for \eqref{eq:nonlocal-AC-multiD} 
and \cite{AlFu,AlFuKa,Pego} for \eqref{eq:CH-multiD}.

In this paper, we are interested in studying the interface motion for some hyperbolic variations of the one-dimensional version of \eqref{eq:nonlocal-AC-multiD} 
and in Sections \ref{sec:st-main}-\ref{sec:layerdyn} we describe in detail the layer dynamics for \eqref{eq:hyp-nonlocal}, 
comparing it with equations \eqref{eq:nonlocal-AC-multiD}, \eqref{eq:AC-multiD} and \eqref{eq:CH-multiD}.

In the next section, we introduce the hyperbolic variation \eqref{eq:hyp-nonlocal} of the mass conserving Allen--Cahn equation.

\subsection{Hyperbolic mass conserving Allen--Cahn equation}
In the previous section, we discussed some properties of the mass conserving Allen--Cahn equation and the link with the classical Allen--Cahn and Cahn--Hilliard equations.
In the past years, hyperbolic variations of the classical versions \eqref{eq:AC-multiD}-\eqref{eq:CH-multiD} 
have been proposed to avoid some unphysical behavior of the solutions.
First, (parabolic) reaction-diffusion equations of the form \eqref{eq:AC-multiD} undergo the same criticism of the linear diffusion equation, 
mainly concerning infinite speed of propagation of disturbances and lack of inertia. 
Hence, following some ideas developed by Maxwell in the context of kinetic theories, 
Cattaneo \cite{Cat} proposed a relaxation law instead of the classic Fourier (or Fick) law, 
leading to a hyperbolic reaction-diffusion equation (see \cite{JP89a,JP89b}, \cite{FLM17} and references therein). 
Second, following the classical Maxwell--Cattaneo modification of the Fick's diffusion law, 
Galenko \cite{Galenko} proposed a hyperbolic relaxation of \eqref{eq:CH-multiD} in order to describe the early stages of spinodal decomposition
in certain glasses (among others see \cite{FLMpre} and reference therein).

Here, following the same ideas of \cite{Cat} and \cite{Galenko},
we consider a hyperbolic variant of equation \eqref{eq:nonlocal-AC-multiD}, 
which is obtained by using the Maxwell-Cattaneo law, instead of the classic Fick law.
A generic reaction-diffusion equation of the form \eqref{eq:nonlocal-AC-multiD} can be obtained from the continuity equation
\begin{equation}\label{eq:continuity}
	u_t+\nabla\cdot\bm v=f(u)-\lambda_f,
\end{equation}
where $\bm v$ is the flux of $u$, and the Fick (or Fourier) law
\begin{equation}\label{eq:Fick}
	\bm v=-\nabla u.
\end{equation}
By substituting \eqref{eq:Fick} into \eqref{eq:continuity}, one obtains equation \eqref{eq:nonlocal-AC-multiD}.
Therefore, equation \eqref{eq:nonlocal-AC-multiD} is a consequence of the instantaneous equilibrium between the flux $\bm v$ and $-\nabla u$ given by \eqref{eq:Fick}.
On the other hand, one can think that such equilibrium is not instantaneous but delayed, namely we assume that there exists $\tau>0$ such that
\begin{equation*}
	\bm v(\bm x,t+\tau)=-\nabla u(\bm x,t), \qquad \qquad \forall\,\bm x\in\Omega, \, t>0.
\end{equation*}
By taking $\bm v+\tau \bm v_t$ as first approximation of $\bm v(t+\tau)$, we obtain the \emph{Maxwell--Cattaneo law}
\begin{equation}\label{eq:Max-Cat}
	\tau\bm v_t+\bm v=-\nabla u, \qquad \qquad \tau>0,
\end{equation}
which has been proposed to describe heat propagation by conduction with finite speed \cite{Cat}, \cite{JP89a,JP89b}.
Indeed, in the case $f=0$, the system \eqref{eq:continuity}-\eqref{eq:Fick} becomes the linear diffusion equation (heat equation) and it is well-known that
it allows infinite speed of propagation of disturbances: a small perturbation in a point $\bm{x_0}$ changes instantaneously 
the solution $u$ in every point $\bm x$ of the domain $\Omega$.
The relaxation law \eqref{eq:Max-Cat} has been proposed in order to avoid this unphysical property and to take in account inertial effects. 
The parameter $\tau$ is a relaxation time and describes the time taken by the flux $\bm v$ to relax to $-\nabla u$.
Using the constitutive equation \eqref{eq:Max-Cat} instead of \eqref{eq:Fick}, we obtain the system
\begin{equation*}
	\begin{cases}
		u_t+\nabla\cdot\bm v=f(u)-\lambda_f,\\
		\tau\bm v_t+\nabla u=-\bm v.
	\end{cases}
\end{equation*}
To obtain a single equation for $u$, let us multiply by $\tau$ and differentiate with respect to time the first equation,
and take the divergence of the second one; 
we deduce the following \emph{mass-conserving reaction-diffusion equation with relaxation}
\begin{equation}\label{eq:hyp-nonlocal-multiD}
	\tau u_{tt}+\left\{u-\tau f(u)+\tau\lambda_f\right\}_t=\Delta u+f(u)-\lambda_f.
\end{equation}
In the rest of the paper, we consider a more general version of \eqref{eq:hyp-nonlocal-multiD} in $[0,1]$:
for $G:\R\to\R$, we consider the equation 
\begin{equation*}
	\tau u_{tt}+\left\{G(u)+\int_0^1\left[u-G(u)\right]\,dx\right\}_t=\e^2u_{xx}+f(u)-\int_0^1f(u)\,dx.
\end{equation*}
Notice that, by expanding the time derivative, one obtains equation \eqref{eq:hyp-nonlocal} with $g=G'$.
The main examples we have in mind are $g\equiv1$, which corresponds to 
\begin{equation*}
	\tau u_{tt}+u_t=\e^2 u_{xx}+f(u)-\int_0^1f(u)\,dx,
\end{equation*}
and the relaxation case $g(u)=1-\tau f'(u)$, which corresponds to
\begin{equation*}
	\tau u_{tt}+\{1-\tau f'(u)\}u_t+\tau\int_0^1f'(u)u_t\,dx= \e^2u_{xx}+f(u)-\int_0^1f(u)\,dx.
\end{equation*}
In the latter case, once the reaction term $f$ is fixed, assumption \eqref{eq:ass-g} imposes a restriction on the parameter $\tau$, which must satisfy
\begin{equation*}
	0<\tau<\frac1{\max f'(u)}.
\end{equation*}
Further details on the laws \eqref{eq:Fick}, \eqref{eq:Max-Cat} and other choices of the damping coefficient $g$, 
corresponding to different modifications of the Fick's law can be found in \cite{LMPS}.

As we will see in Section \ref{sec:st-main}, in general the solutions to the hyperbolic version \eqref{eq:hyp-nonlocal} do not conserve the mass.
However, imposing the following condition on the initial velocity,
\begin{equation}\label{eq:ass-u1}
	\int_0^1u_1(x)\,dx=0,	
\end{equation}
we obtain conservation of the mass and \eqref{eq:hyp-nonlocal} possesses the energy functional
\begin{equation}\label{eq:energy}
	E[u,u_t](t):=\int_0^1\left[\frac\tau2u^2_t(x,t)+\frac{\e^2}2 u_x^2(x,t)+F(u(x,t))\right]\,dx.
\end{equation}
More precisely, the assumptions \eqref{eq:ass-g} and \eqref{eq:ass-u1} imply that if $u$ is a solution to \eqref{eq:hyp-nonlocal} 
with boundary conditions \eqref{eq:Neumann}, then (see Lemma \ref{lem:energy-estimate})
\begin{equation*}
	\frac{d}{dt}E[u,u_t](t)\leq-\sigma\int_0^1 u_t^2(x,t)\,dx.
\end{equation*}
Therefore, when the initial velocity is a function of zero mean, 
we have a hyperbolic reaction-diffusion equation with the property that the total mass is preserved in time 
and with the energy functional \eqref{eq:energy}, which has been used in \cite{JHDE2017} to study hyperbolic reaction-diffusion equations and in \cite{FLM19} to prove
exponentially slow motion for some solutions to a hyperbolic relaxation of the Cahn--Hilliard equation.

In this paper, we assume that \eqref{eq:ass-u1} is satisfied and then we study the metastable dynamics of the solutions when the mass is conserved.
It is worth to stress that, by using the energy functional \eqref{eq:energy} and adapting the procedure of \cite{FLM19},
one can prove the exponentially slow motion of the solutions also without the assumption \eqref{eq:ass-u1} (see Section \ref{sec:energyapp}).
On the contrary, the strictly positiveness of the damping coefficient $g$ \eqref{eq:ass-g} is crucial
because it guarantees the dissipative character of equation \eqref{eq:hyp-nonlocal}.

We conclude this Introduction with a short presentation of the main results of this paper.
First of all, we shall prove that there exists an \emph{approximately invariant manifold} $\mathcal{M}_{{}_0}$ for the IBVP \eqref{eq:hyp-nonlocal}-\eqref{eq:Neumann}-\eqref{eq:initial}.
Precisely, the manifold $\mathcal{M}_{{}_0}$ is not invariant, but we will construct a tubular neighborhood (slow channel) of $\mathcal{M}_{{}_0}$ satisfying the following property:
any solution to \eqref{eq:hyp-nonlocal}-\eqref{eq:Neumann}-\eqref{eq:initial} starting from such a slow channel can leave it 
only after an exponentially long time, i. e. a time of $\mathcal{O}(\exp(C/\e))$ as $\e\to0^+$.
Moreover, inside the slow channel the solution is a function with a finite number ($N>1$) of transitions between the minimum points $\pm1$ of the potential $F$;
we shall derive a system of ODEs which describes the motion of the layers inside the slow channel, and as a consequence
the dynamics of the solution to \eqref{eq:hyp-nonlocal}-\eqref{eq:Neumann}-\eqref{eq:initial}.
Summarizing, we shall prove that the phenomenon of metastability is also present in the case of \eqref{eq:hyp-nonlocal}-\eqref{eq:Neumann}:
some solutions maintain for a very long time an unstable structure with $N>1$ transitions and we describe in detail the exponentially slow motion of the layers.

The approach we used here can be also adapted to study the mass conserving Allen--Cahn equation \eqref{eq:maco-AC} 
in order to obtain similar results on the metastable dynamics of the solutions:
existence of an approximately invariant manifold and derivation of the ODEs for the layers.
To the best of our knowledge, the only papers devoted to the metastability for the mass conserving Allen--Cahn equation \eqref{eq:maco-AC} 
are \cite{ReyWar,SunWard}, where the authors use formal asymptotic methods and impose the conservation of mass
to derive a system of ODEs describing the layer dynamics for \eqref{eq:maco-AC}. 
Then, they compare these asymptotic results with corresponding full numerical results.
As we will see in Sections \ref{sec:st-main} and \ref{sec:layerdyn}, by using a different approach, 
we derive a system of ODEs describing the layer dynamics for \eqref{eq:hyp-nonlocal}
and in the limit $\tau\to0^+$, $g\to1$, we obtain the same system of \cite{ReyWar,SunWard}.

The rest of the paper is organized as follows.
In Section \ref{sec:st-main} we present our main results.
First, we state Theorem \ref{thm:main}, which establishes the existence of a slow channel for \eqref{eq:hyp-nonlocal}-\eqref{eq:Neumann} 
and, as a consequence, the existence of an approximately invariant manifold $\mathcal{M}_{{}_0}$ for \eqref{eq:hyp-nonlocal}-\eqref{eq:Neumann}.
Second, we present the system of ODEs which describes the motion of the layers.
In Section \ref{sec:base}, we collect some preliminary results needed to prove our main results;
in particular, we introduce a new system of coordinates for functions close to the manifold $\mathcal{M}_{{}_0}$. 
Finally, Section \ref{sec:slow} contains the proof of Theorem \ref{thm:main} and 
in Section \ref{sec:layerdyn} we derive the ODEs describing the layer dynamics.

\section{Main results}\label{sec:st-main}
The goal of this section is to present the main results of the paper.
Before doing this, we prove some properties of the solution to the IBVP \eqref{eq:hyp-nonlocal}-\eqref{eq:Neumann}-\eqref{eq:initial}, 
valid for a generic reaction term $f$, which are consequences of the assumption \eqref{eq:ass-u1}.
Moreover, we present some energy estimates, which permit to obtain persistence of metastable patterns for an exponentially long time as $\e\to0^+$, in the case of a balanced bistable reaction term, i.e. a reaction term $f=-F'$ with $F$ satisfying \eqref{eq:ass-F}.

\subsection{Mass conservation and energy estimates}\label{sec:energyapp}
By integrating \eqref{eq:hyp-nonlocal} in $[0,1]$ and using the homogeneous Neumann boundary conditions \eqref{eq:Neumann},
we deduce the following ODE for the mass $m(t):=\displaystyle\int_0^1 u(x,t)\,dx$:
\begin{equation}\label{eq:ODEmass}
	\tau m''(t)+m'(t)=0, \qquad m(0)=\int_0^1 u_0(x)\,dx, \qquad m'(0)=\int_0^1u_1(x)\,dx,
\end{equation}
and, as a consequence, $m(t)=m(0)+\tau m'(0)(1-\exp(-t/\tau))$.
It follows that the mass is conserved, i.e. $m(t)\equiv m(0)$, if and only if \eqref{eq:ass-u1} holds.

Another consequence of the assumption \eqref{eq:ass-u1} is that if $g$ is a strictly positive function \eqref{eq:ass-g},
then the energy defined in \eqref{eq:energy} is a non-increasing function of $t$ along the solutions to \eqref{eq:hyp-nonlocal}-\eqref{eq:Neumann}.
Precisely, we have the following energy estimates.
\begin{lem}\label{lem:energy-estimate}
Assume that $g$ satisfies \eqref{eq:ass-g}.
If $(u,u_t)\in C\left([0,T],H^2(0,1)\times H^1(0,1)\right)$ is solution to \eqref{eq:hyp-nonlocal}-\eqref{eq:Neumann}-\eqref{eq:initial} for some $T>0$, 
with $u_1$ satisfying \eqref{eq:ass-u1}, then
\begin{equation}\label{eq:energy-dissipation}
	\frac{d}{dt} E[u,u_t](t)\leq-\sigma\int_0^1 u_t^2(x,t)\,dx, 
\end{equation}
for any $t\in[0,T]$.
\end{lem}
\begin{proof}
By differentiating with respect to $t$ the definition \eqref{eq:energy} and integrating by parts, we infer
\begin{equation*}
	\frac{d}{dt} E[u,u_t](t)=\int_0^1 u_t(x,t)\left[\tau u_{tt}(x,t)-\e^2u_{xx}(x,t)-f(u(x,t))\right]\,dx,
\end{equation*}
where we used the homogeneous Neumann boundary conditions \eqref{eq:Neumann} and the fact that $F'=-f$.
Since $u$ is a solution to \eqref{eq:hyp-nonlocal}, we have
\begin{equation}\label{eq:energy-assu1}
	\begin{aligned}
		\frac{d}{dt} E[u,u_t](t)=&-\int_0^1 g(u(x,t))u_t(x,t)^2\,dx\\
		&-m'(t)\int_0^1\Big\{\big[1-g(u(x,t))\big]u_t(x,t)+f(u(x,t))\Big\}\,dx,
	\end{aligned}
\end{equation}
and the estimate \eqref{eq:energy-dissipation} follows from the assumptions \eqref{eq:ass-g}-\eqref{eq:ass-u1} and \eqref{eq:ODEmass}.
\end{proof}

\begin{rem}
In Lemma \ref{lem:energy-estimate}, we assume that there exists a sufficiently smooth solution to \eqref{eq:hyp-nonlocal}-\eqref{eq:Neumann}-\eqref{eq:initial} 
and we prove the estimate \eqref{eq:energy-dissipation}.
Studying the well-posedness of the IBVP \eqref{eq:hyp-nonlocal}-\eqref{eq:Neumann}-\eqref{eq:initial} is beyond the scope of this paper 
and in the following we assume that there exists a sufficiently smooth solution.
However, in the case of a strictly positive damping coefficient \eqref{eq:ass-g} and with initial velocity of zero-mean \eqref{eq:ass-u1}, 
one can extend to the IBVP \eqref{eq:hyp-nonlocal}-\eqref{eq:Neumann}-\eqref{eq:initial} the well-posedness results of \cite[Appendix A]{JHDE2017}.
\end{rem}

Thanks to the dissipative estimate \eqref{eq:energy-dissipation},
one can prove existence of metastable patterns for the boundary problem \eqref{eq:hyp-nonlocal}-\eqref{eq:Neumann},
by using the energy approach firstly introduced in \cite{Bron-Kohn} to study the classical Allen--Cahn equation
\begin{equation}\label{eq:AC}
	u_t=\e^2u_{xx}+f(u),
\end{equation}
and then successfully applied to different models, like the \emph{hyperbolic Allen--Cahn equation}
\begin{equation}\label{eq:hypAC}
	\tau u_{tt}+g(u)u_t=\e^2u_{xx}+f(u),
\end{equation}
and the \emph{hyperbolic Cahn--Hilliard equation}
\begin{equation}\label{eq:hypCH}
	\tau u_{tt}+u_t=-\left(\e^2u_{xx}+f(u)\right)_{xx},
\end{equation}
for details see \cite{JHDE2017,FLM19} and references therein.
In the following, we briefly explain the strategy of such energy approach and how 
to apply it to the IBVP \eqref{eq:hyp-nonlocal}-\eqref{eq:Neumann}-\eqref{eq:initial} when $F$ satisfies \eqref{eq:ass-F}.
Multiplying by $\e^{-1}$ and integrating \eqref{eq:energy-dissipation} in $[0,T]$, for any $T>0$, we deduce the estimate
\begin{equation}\label{eq:energy-variation}
	\sigma\e^{-1}\int_0^T\!\int_0^1u_t^2(x,t)\,dxdt\leq E_\e[u_0,u_1]-E_\e[u,u_t](T),
\end{equation}
where $E_\e$ is the renormalized energy 
\begin{equation*}
	E_\e[u,u_t](t):=\frac{1}{\e}E[u,u_t](t):=\int_0^1\left[\frac\tau{2\e}u^2_t(x,t)+\frac{\e}2 u_x^2(x,t)+\frac{F(u(x,t))}\e\right]\,dx.
\end{equation*}
The main idea of the energy approach \cite{Bron-Kohn} is to derive an estimate for the $L^2$--norm of the time derivative $u_t$ from \eqref{eq:energy-variation} when $T\gg1$;
then, we need an \emph{upper bound} on $E_\e[u_0,u_1]$ and a \emph{lower bound} on $E_\e[u,u_t](T)$ for some $T$ very large when $\e\to0^+$.
For the upper bound, we can properly choose the initial datum $(u_0^\e,u_1^\e)$ (depending on $\e$):
fix $N\in\mathbb{N}$, $0<h_1<\dots<h_{N+1}<1$ and assume that
\begin{equation}\label{eq:ass-energyapp}
	\lim_{\e\to0}\|u_0^\e-v\|_{{}_{L^1}}=0, \qquad \qquad E_\e[u_0^\e,u_1^\e]\leq (N+1)c_{{}_F}+C_1\exp(-C_2/\e),
\end{equation}
where $v:[0,1]\to\{-1,+1\}$ is a step function with exactly $N+1$ jumps at $h_1<\dots<h_{N+1}$, the constants $C_1,C_2$ are strictly positive and independent on $\e$, and
\begin{equation}\label{eq:c_F}
	c_{{}_F}:=\displaystyle\int_{-1}^1\sqrt{2F(s)}\,ds
\end{equation}
represents the minimum energy to have a transition between $-1$ and $+1$ \cite{Bron-Kohn,JHDE2017,FLM19}.
An example of initial data satisfying \eqref{eq:ass-energyapp} can be found in \cite{FLM19}.
Concerning the lower bound, it could be obtained by proceeding as in \cite{JHDE2017,FLM19}, because the energy functional $E_\e$ is the same.
In particular, the lower bound is a consequence of a variational result on the Ginzburg--Landau functional
\begin{equation*}
	\int_0^1\left[\frac{\e}2 u_x^2+\frac{F(u)}\e\right]\,dx,
\end{equation*}
and it reads as
\begin{equation*}
	E_\e[u,u_t](\e^{-1}T_\e)\geq(N+1)c_{{}_F}-C_3\exp(-C_2/\e),
\end{equation*}
where $T_\e=\mathcal{O}(\exp(C_2/\e))$.
Substitution of the latter lower bound and assumption \eqref{eq:ass-energyapp} in the key estimate \eqref{eq:energy-variation} yields the bound
\begin{equation}\label{eq:ut-energyapp}
	\int_0^{\e^{-1}T_\e}\!\!\int_0^1u_t^2(x,t)\,dxdt\leq C\e\exp(-C_2/\e),
\end{equation}
which permits to prove that some solutions to \eqref{eq:hyp-nonlocal}-\eqref{eq:Neumann} maintain the same structure of the initial datum 
for the time $T_\e$ as $\e\to0^+$, for details see \cite{Bron-Kohn,JHDE2017,FLM19}.
We stress again that the key point of the energy approach is the estimate \eqref{eq:energy-dissipation}, which implies \eqref{eq:energy-variation}.
It is worth to notice that the energy approach also works when the assumption \eqref{eq:ass-u1} on $u_1$ is not satisfied.
For simplicity, consider the case $g\equiv1$; from \eqref{eq:energy-assu1} it follows that 
\begin{equation*}
	\e^{-1}\int_0^T\!\int_0^1u_t^2(x,t)\,dxdt\leq E_\e[u_0,u_1]-E_\e[u,u_t](T)+C\|u_1\|_{{}_{L^1}}, \qquad \qquad \forall\,T>0,
\end{equation*}
and then, the estimate \eqref{eq:ut-energyapp} could be obtained as in \cite{JHDE2017,FLM19} by using the fact that $\|u_1\|_{{}_{L^1}}=\mathcal{O}(\exp(-C_2/\e))$.

\subsection{Approximately invariant manifold}
The goal of this paper is to study the metastable dynamics of the solutions to \eqref{eq:hyp-nonlocal}-\eqref{eq:Neumann}-\eqref{eq:initial},
by using the dynamical approach proposed by Carr--Pego \cite{Carr-Pego} and Fusco--Hale \cite{Fusco-Hale} 
to describe the metastable dynamics of the solutions to \eqref{eq:AC} and then applied to the Cahn--Hilliard equation in \cite{Bates-Xun1,Bates-Xun2},
and to the hyperbolic variants \eqref{eq:hypAC}, \eqref{eq:hypCH} in \cite{FLM17}, \cite{FLMpre}, respectively.
To start with, we introduce some notations and definitions.

In all the paper we denote by $\|\cdot\|$ and $\langle\cdot,\cdot\rangle$ the norm and inner product in $L^2(0,1)$.
Moreover, in what follows we fix $N\in\mathbb{N}$ and define for $\rho>0$ the set of admissible layer positions 
\begin{align*}
	\Omega_\rho:=\Bigl\{\bm h\in\R^{N+1} \,:\, 0<h_1<\dots<h_{N+1}<1, \, \mbox{ and } \, h_{j+1}&-h_j>\e/\rho, \\ 
	&\mbox{ for } j=0,\dots,N+1\Bigr\},
\end{align*}
where $h_0=-h_1$ and $h_{N+2}=2-h_{N+1}$, because of the homogeneous Neumann boundary conditions \eqref{eq:Neumann}.
Finally, we fix $\delta\in(0,1/N+1)$, consider the parameters $\e$ and $\rho$ such that
\begin{equation}\label{eq:triangle}
	\e\in(0,\e_0) \qquad \mbox{ and } \qquad \delta<\frac{\e}{\rho}<\frac{1}{N+1}, 
\end{equation}
for some $\e_0>0$ to be chosen appropriately small and we introduce the $(N+1)$--manifold
\begin{equation}\label{eq:M^AC}
	\mathcal{M}^{AC}:=\{u^{\bm h} :\bm h\in\Omega_\rho\},
\end{equation} 
where $u^{\bm h}$ is a function with $N+1$ transitions, which approximates a metastable patterns with layers at $h_1, \dots,h_{N+1}$.
The construction of $u^{\bm h}$ was introduced in \cite{Carr-Pego} and since 
the metastable states are the same for the equations \eqref{eq:AC}, \eqref{eq:hypAC} and \eqref{eq:hypCH}, 
it was also used in \cite{Bates-Xun1,Bates-Xun2}, \cite{FLM17} and \cite{FLMpre}.
We give the precise definition of $u^{\bm h}$ in Section \ref{sec:base};
here we recall that $u^{\bm h}$ is approximately $\pm1$ except to an $\mathcal{O}(\e)$-neighborhood of $h_1,\dots,h_{N+1}$, namely
\begin{equation}\label{eq:uh-approx}
	u^{\bm h}(x)\approx(-1)^j, \quad \mbox{for } x\in[h_{j-1}+\mathcal{O}(\e),h_j-\mathcal{O}(\e)]\cap[0,1] \quad \mbox{ and } \quad j=1,\dots,N+2,
\end{equation} 
and $u^{\bm h}$ is well approximated by standing waves solutions to \eqref{eq:AC} in the $\mathcal{O}(\e)$-neighborhood of $h_j$ (for details see \cite[Proposition 2.2]{Carr-Pego}).

In \cite{Carr-Pego}, the authors show that the manifold $\mathcal{M}^{AC}$ is approximately invariant for the Allen--Cahn equation \eqref{eq:AC}, 
while in \cite{FLM17} it is proved that the \emph{extended} manifold
\begin{equation*}
	\mathcal{M}^{AC}_{{}_0}:=\mathcal{M}^{AC}\times\{0\}=\{(u^{\bm h},0) :u^{\bm h}\in\mathcal{M}^{AC}\}
\end{equation*}
is approximately invariant for the hyperbolic variant \eqref{eq:hypAC}.
The mass conservation allows us to work with the manifolds
\begin{equation}\label{eq:M_0}
	\mathcal{M}:=\left\{u^{\bm h}\in \mathcal{M}^{AC} : \, \int_0^1u^{\bm h}(x)\,dx=M\right\}, \qquad\;
	\mathcal{M}_{{}_0}:=\{(u^{\bm h},0) :u^{\bm h}\in\mathcal{M}\},
\end{equation} 
where $M\in(-1,1)$ represents the mass of the solution (the mass of the initial datum $u_0$).
The manifolds $\mathcal{M}$ and $\mathcal{M}_{{}_0}$ are approximately invariant for the Cahn--Hilliard equation (see \cite{Bates-Xun1}) 
and its hyperbolic variant \cite{FLMpre}, respectively.
Our goal is to prove that the \emph{base manifold} $\mathcal{M}_{{}_0}$ is also approximately invariant for \eqref{eq:hyp-nonlocal}-\eqref{eq:Neumann}.
As we already mentioned, the fact that $\mathcal{M}$ is approximately invariant for \eqref{eq:maco-AC}
has not been proved in literature, but it can be proved with the approach we used here.

To prove that $\mathcal{M}_{{}_0}$ is approximately invariant for equation \eqref{eq:hyp-nonlocal}, 
we shall construct a tubular neighborhood $\mathcal{Z}_{{}_{\rho}}$ of $\mathcal{M}_{{}_0}$ (see definition \eqref{eq:slowchannel})
and we prove that if the initial datum $(u_0,u_1)\in\stackrel{\circ}{\mathcal{Z}}_{{}_{\rho}}$, 
then the corresponding solution to the IBVP \eqref{eq:hyp-nonlocal}-\eqref{eq:Neumann}-\eqref{eq:initial} can leave $\mathcal{Z}_{{}_{\rho}}$ 
only after an exponentially long time.

\begin{thm}\label{thm:main}
Let $f\in C^2(\R)$ and $g\in C^1(\R)$ be such that $f=-F'$ and \eqref{eq:ass-g}-\eqref{eq:ass-F} hold.
Given $N\in\mathbb{N}$ and $\delta\in(0,1/N+1)$, there exist $\varepsilon_0>0$ and a slow channel $\mathcal{Z}_{{}_{\rho}}$ 
containing $\mathcal{M}_{{}_{0}}$, such that if $\varepsilon,\rho$ satisfy \eqref{eq:triangle},
and the initial datum satisfies $(u_0,u_1)\in\,\stackrel{\circ}{\mathcal{Z}}_{{}_{\rho}}$,
then the solution $(u,u_t)$ to the initial-boundary value problem \eqref{eq:hyp-nonlocal}-\eqref{eq:Neumann}-\eqref{eq:initial}
remains in $\mathcal{Z}_{{}_{\rho}}$ for a time $T_\varepsilon>0$, and there exists $C>0$ such that
for any $t\in[0,T_\varepsilon]$,
\begin{align}
	\varepsilon^{1/2}\|u-u^{\bm h}\|_{{}_{L^\infty}}+\|u-u^{\bm h}\|+\tau^{1/2}\|u_t\|&\leq C\exp(-A\ell^{\bm h}/\varepsilon), \label{eq:umenouh}\\
	|{\bm h}'|_{{}_{\infty}} &\leq C\left(\varepsilon/\tau\right)^{1/2}\exp(-A\ell^{\bm h}/\varepsilon), \label{eq:|h'|<exp-intro}
\end{align}
where $A:=\sqrt{\min\{F''(-1),F''(1)\}}$, $\ell^{\bm h}:=\min\{h_j-h_{j-1}\}$ and $|\cdot|_{{}_{\infty}}$
denotes the maximum norm in $\mathbb{R}^N$.
Moreover,
\begin{equation*}
		T_\varepsilon\geq C\left(\tau/\varepsilon\right)^{1/2}(\ell^{\bm h(0)}-\varepsilon/\rho)\exp(A\delta /\varepsilon).
\end{equation*}
\end{thm}
Thanks to the estimate \eqref{eq:umenouh} and the lower bound on $T_\e$ we can say that, for an exponentially long time, 
the solution $u$ to the IBVP \eqref{eq:hyp-nonlocal}-\eqref{eq:Neumann}-\eqref{eq:initial} is well approximated by $u^{\bm h}\in\mathcal{M}$ 
and the $L^2$--norm of the time derivative $u_t$ is exponentially small as $\e\to0^+$.
Therefore, $u$ is a function with $N+1$ layers satisfying \eqref{eq:uh-approx}
and \eqref{eq:|h'|<exp-intro} ensures that the layers move with an exponentially small velocity.

Now, we briefly explain the strategy to prove Theorem \ref{thm:main}.
As in the case of the Cahn--Hilliard equation, we work with a different variable describing the position of the layers;
indeed, the manifold $\mathcal{M}$ is a constant mass submanifold of the $(N+1)$--manifold $\mathcal{M}^{AC}$ (cfr. definitions \eqref{eq:M^AC}, \eqref{eq:M_0}),
and it can be parametrized by the first $N$ components of the vector $\bm h$, for details see \cite[Lemma 2.1]{Bates-Xun1}.
Then, we introduce the vector $\bm\xi=(h_1,\dots,h_N)$ consisting of the first $N$ components of $\bm h$
and we denote by $u^{\bm\xi}$ an element of $\mathcal{M}$.

Next, we introduce the decomposition $u=u^{\bm\xi}+w$, where the remainder $w$ is orthogonal to appropriate functions $\nu_j^{\bm\xi}$, i. e.
\begin{equation}\label{eq:ortho-sec2}
	\langle w,\nu^{\bm\xi}_j\rangle=0, \qquad \qquad j=1,\dots,N.
\end{equation}
The choice of the functions $\nu_j^{\bm\xi}$ is crucial in our work and, as we will see in the definition \eqref{eq:newtangvec}, 
they are linear combinations of the approximate tangent vectors of $\mathcal{M}^{AC}$ introduced in \cite{Carr-Pego}. 
Then, we prove that for any function $u$ having mass equal to $M$ and belonging to a small neighborhood of $\mathcal{M}$, 
there exists a unique $u^{\bm\xi}\in\mathcal{M}$ (then, having the same mass of $u$) such that
$u=u^{\bm\xi}+w$, with $w$ satisfying the orthogonality condition \eqref{eq:ortho-sec2}, for details see Theorem \ref{thm:existence-coord}.
Therefore, we extend to the constant mass submanifold $\mathcal{M}$ the results valid for the $(N+1)$--manifold $\mathcal{M}^{AC}$ \cite{Carr-Pego}
and, as we will see in Section \ref{sec:slow}, such a decomposition plays a crucial role in the proof of Theorem \ref{thm:main}.
In particular, we derive an ODE-PDE coupled system \eqref{eq:system-w-v-xi} for the new coordinates $(\bm\xi,w)$ and study it in an appropriate slow channel.
By using some energy estimates, we prove that in $\mathcal{Z}_{{}_{\rho}}$ the estimates \eqref{eq:umenouh}-\eqref{eq:|h'|<exp-intro} hold 
and the solution $u$ leaves $\mathcal{Z}_{{}_{\rho}}$ if and only if $\bm h\in\partial\Omega_\rho$, 
meaning that $h_{j+1}-h_j=\e/\rho$ for some $j\in1,\dots,N+1$ (two transition points are close enough).
Since the layers move with an exponentially small velocity, the time taken for the solution to leave $\mathcal{Z}_{{}_{\rho}}$ is exponentially large.
\begin{rem}\label{rem:main-tau}
The appearance of the relaxation parameter $\tau>0$ in \eqref{eq:|h'|<exp-intro} and in the lower bound for $T_\e$ is a consequence of the estimate \eqref{eq:umenouh}.
Indeed, as we already mentioned, we first prove that in the slow channel the solution satisfies \eqref{eq:umenouh}-\eqref{eq:|h'|<exp-intro};
in particular, the velocity of the layers can be bounded by the quantity $\|u_t\|$, cfr. Proposition \ref{prop:E>}, 
and as a consequence, $\tau$ appears in the denominator of the right hand side of \eqref{eq:|h'|<exp-intro} and in the lower bound for $T_\e$, 
that is inversely proportional to the velocity of the layers.
Such a way to obtain the exponentially small velocity of the layers is due to the \emph{hyperbolic} character of the equation \eqref{eq:hyp-nonlocal} 
(the presence of the inertial term $\tau u_{tt}$); 
in the case of the classic Allen--Cahn, Cahn--Hilliard and mass conserving Allen--Cahn equations, 
the exponentially small velocity could be directly obtained from the ODEs for the layers, 
without using estimates on $\|u_t\|$ (cfr. \cite{Carr-Pego}, \cite{Bates-Xun1} and Remark \ref{rem:energy-tau}).
\end{rem}

\subsection{ODE for the layers}
After proving Theorem \ref{thm:main}, in Section \ref{sec:layerdyn} we derive the system of ODEs
describing the layer dynamics, which read as 
\begin{equation}\label{eq:ODE-hypnonlocal}
	\tau h''_j+\gamma_{{}_{F,g}} h'_j=\frac{\e}{c_{{}_F}}\left(\alpha^{j+1}-\alpha^j+\frac{(-1)^{j+1}}{N+1}\sum_{i=1}^{N+1}(-1)^i(\alpha^{i+1}-\alpha^i)\right),
\end{equation}
for $j=1,\dots,N+1$, where $\gamma_{{}_{F,g}}$ is a positive constant depending only on $F$ and $g$ (see definition below), $c_{{}_F}$ is defined in \eqref{eq:c_F} and $\alpha^j$ depends on $\e$, $F$ and $\bm h$.
In particular, the term $\alpha^{j+1}-\alpha^j$ determines the speed of the transition point $h_j$ 
in the case of the classical Allen--Cahn equation \eqref{eq:AC} (see \cite[Section 6]{Carr-Pego}), that is 
\begin{equation}\label{eq:ODE-AC}
	h'_j=\frac{\e}{c_{{}_F}}\left(\alpha^{j+1}-\alpha^j\right), \qquad \qquad j=1,\dots,N+1.
\end{equation}
The system \eqref{eq:ODE-AC}, which describes the layer dynamics in the case of \eqref{eq:AC}, has been derived and studied in detail in \cite[Section 6]{Carr-Pego};
here, we stress that the velocity of $h_j$ is exponentially small and 
depends only from the distance to the nearest layers $h_{j-1}$ and $h_{j+1}$.
Precisely, we recall (see Proposition \ref{prop:alfa,beta}) that if $F$ is an even function, then
\begin{equation*}
	\alpha^j=K\exp\left(-\frac{Al_j}{\e}\right)\left\{1+\mathcal{O}\left(\e^{-1}\exp\left(-\frac{Al_j}{2\e}\right)\right)\right\},  \qquad \qquad j=1,\dots,N+1,
\end{equation*}
for some $K>0$, where $A:=\sqrt{F''(\pm1)}$ and $l_j:=h_{j+1}-h_j$.
Hence, the layer dynamics of \eqref{eq:AC} is described by the ODEs
\begin{equation*}
	h'_j=\frac{\e K}{c_{{}_F}}\left[\exp\left\{-\frac{A(h_{j+1}-h_j)}{\e}\right\}-\exp\left\{-\frac{A(h_j-h_{j-1})}{\e}\right\}\right],
\end{equation*}
for $j=1,\dots,N+1$.
Moreover, one has
\begin{equation*}
	\frac{\alpha^j}{\alpha^i}\leq C\exp\left(-\frac{A}{\e}(l_j-l_i)\right),  
\end{equation*}
for some $C>0$, and if $l_j-l_i\geq \kappa$ for some $\kappa>0$, we deduce
\begin{equation*}
	\alpha^j\leq C\exp\left(-\frac{A\kappa}{\e}\right)\alpha^i.
\end{equation*}
Therefore, if $l_j>l_i$ then $\alpha^j<\alpha^i$, and for $\e/\kappa\ll1$, $\alpha^j$ is \emph{exponentially small} with respect to $\alpha^i$.
Such properties of $\alpha^j$ allow us to briefly describe the layer dynamics for \eqref{eq:AC} as follows.
For simplicity, assume that there exists a unique $i\in\{1,\dots,N\}$ such that
\begin{equation}\label{eq:ass-i}
	h_{i+1}-h_i<h_{j+1}-h_j, \qquad \qquad j\neq i,\quad j=0,\dots,N+1, 
\end{equation}
meaning that $h_i$ and $h_{i+1}$ are the closest layers for some $i\neq0,N+1$.
In this case, $h_i$ and $h_{i+1}$ move towards each other with approximately the same speed and the other $N-2$ points are essentially static, 
being $\alpha^{i+1}\gg\alpha^j$ for $\e\ll1$ and $j\neq i+1$.

In the case of equation \eqref{eq:maco-AC}, the situation is different because of the mass conservation.
Taking (formally) the limit as $\tau\to0^+$ and $\gamma_{{}_{F,g}}\to1$ in \eqref{eq:ODE-hypnonlocal}, we found the ODEs
\begin{equation}\label{eq:ODE-nonlocalAC}
	h'_j=\frac{\e}{c_{{}_F}}\left(\alpha^{j+1}-\alpha^j+ \frac{(-1)^{j+1}}{N+1}\sum_{i=1}^{N+1}(-1)^i(\alpha^{i+1}-\alpha^i)\right), \qquad j=1,\dots,N+1,
\end{equation}
which describe the dynamics in the case of the mass conserving Allen--Cahn equation \eqref{eq:maco-AC} and was originally proposed in \cite{ReyWar,SunWard}.
Therefore, in \eqref{eq:ODE-nonlocalAC} we have new terms with respect to \eqref{eq:ODE-AC}, 
which take into account the effects of the mass conservation and change notably the motion of the layer.
Indeed, let us assume for definiteness that \eqref{eq:ass-i} holds, $F$ is an even function as above, 
and compare equations \eqref{eq:ODE-AC}-\eqref{eq:ODE-nonlocalAC}: 
we have that the biggest term $\alpha^i$ appears in $h'_j$ for any $j=1,\dots,N+1$ in \eqref{eq:ODE-nonlocalAC}, 
and so, all the layers approximately move with the same exponentially small velocity as $\e\to0^+$.
This is in contrast with \eqref{eq:ODE-AC}, where (as it was already mentioned) the two closest layers move towards each other and the other points are essentially static.
For instance, in the case $N=1$ ($2$ layers), \eqref{eq:ODE-nonlocalAC} becomes
\begin{equation*}
 	h'_1=h'_2=\frac{\e}{2c_{{}_F}}\left(\alpha^3-\alpha^1\right),
\end{equation*}
and the two layers move together in an almost rigid way, that is they move in the same direction at the same speed.
Precisely, $h_1$ and $h_2$ move to the right if and only if $\alpha^3>\alpha^1$, meaning that $1-h_2<h_1$.
In case $N=1$, the layer dynamics is very similar to the one of the Cahn--Hilliard equation, see \cite{Bates-Xun2} or \cite{FLMpre}. 
We stress that the dynamics is very different with respect to the Allen--Cahn equation \eqref{eq:AC};
indeed, for $N=1$ \eqref{eq:ODE-AC} becomes
\begin{equation*}
	h'_1=\frac{\e}{c_{{}_F}}\left(\alpha^2-\alpha^1\right), \qquad\qquad 
	h'_2=\frac{\e}{c_{{}_F}}\left(\alpha^3-\alpha^2\right),
\end{equation*}
and the layers either move towards each other with speed approximately given by $\e\alpha^2$ (if $h_2-h_1<2\min\{h_1,1-h_2\}$) 
or one of the two layers moves towards the closest boundary point ($0$ or $1$) and the other one is essentially static for $\e$ very small.

In the case $N=2$ (3 layers), \eqref{eq:ODE-nonlocalAC} becomes
\begin{align*}
	h'_1&=\frac{\e}{3c_{{}_F}}\left(-2\alpha^1+\alpha^2+2\alpha^3-\alpha^4\right),\\
	h'_2&=\frac{\e}{3c_{{}_F}}\left(-\alpha^1-\alpha^2+\alpha^3+\alpha^4\right),\\
	h'_3&=\frac{\e}{3c_{{}_F}}\left(\alpha^1-2\alpha^2-\alpha^3+2\alpha^4\right),
\end{align*}
and we have 3 points moving with approximately the same speed as $\e\to0^+$; 
precisely, two points move with speed satisfying $|h'_i|\approx\e\alpha^j$ for some $j\in\{1,2,3\}$ and 
the speed $v$ of the third one satisfy $|v|\approx2\e\alpha^j$ as $\e\to0^+$.
This is very different from the layer dynamics of the classical  Allen--Cahn equation, described by \eqref{eq:ODE-AC}, 
and the Cahn--Hilliard equation \cite{Bates-Xun2,FLMpre}, described by 
\begin{align*}
	h'_1&=\frac{1}{4(h_2-h_1)}\left(\alpha^3-\alpha^1\right), \\ 
	h'_2&=\frac{1}{4(h_2-h_1)}\left(\alpha^3-\alpha^1\right)+\frac{1}{4(h_3-h_2)}\left(\alpha^4-\alpha^2\right),\\
	h'_3&=\frac{1}{4(h_3-h_2)}\left(\alpha^4-\alpha^2\right).
\end{align*}
Indeed, for the classical Allen--Cahn equation we have either one point moving towards the closest boundary point and 
the other two essentially static or two points moving towards each other and the third one essentially static;
for the Cahn--Hilliard equation, we have two transitions points moving in the same direction 
at approximately the same speed and the third one is essentially static as $\e\to0^+$.

To conclude this comparison between the layer dynamics of the Allen--Cahn, Cahn--Hilliard and mass-conserving Allen--Cahn equations, 
we recall \cite{Bates-Xun2,FLMpre} that, in the case of the Cahn--Hilliard equation with $N\geq3$ and condition \eqref{eq:ass-i} satisfied with $i\in\{2,\dots,N-1\}$, 
we have four points moving at approximately the same speed, while all the other layers remain essentially stationary in time.
Precisely, we have
\begin{equation*}
	h'_{i-1}>0, \; h'_i>0, \; h'_{i+1}<0, \; h'_{i+2}<0, \; h'_j=\mathcal{O}(e^{-C/\e}h'_i) \; \mbox{ for } j\notin\{i-1,i,i+1,i+2\}, 
\end{equation*} 
and so, the closest layers move towards each other, each being followed by its nearest transition point from ``behind'', 
at approximately the same speed, until the points $h_i$ and $h_{i+1}$ are close enough. 
Hence, the loss of the mass due to the annihilation of the transitions at $h_i$ and $h_{i+1}$ is compensated by the movement of the nearest neighbors $h_{i-1}$ and $h_{i+2}$.
This is the main difference with respect to the mass conserving Allen--Cahn equation, where the loss of the mass is compensated by the movements of all the layers.
There are two cases when the layer dynamics of the mass conserving Allen--Cahn and the Cahn--Hilliard equations are similar:
the previously mentioned case with 2 layers and when we have 4 layers with the closest ones $h_2$ and $h_3$.
Indeed, in such a case, we have 4 layers approximately moving at the same speed in both the mass conversing Allen--Cahn and Cahn--Hilliard equations.
Therefore, we conclude that, under assumption \eqref{eq:ass-i}, 
the layer dynamics in the case of equation \eqref{eq:maco-AC} is always different with respect to equation \eqref{eq:AC}, 
while it is similar to the Cahn--Hilliard equation only in the case of 2 layers and 4 layers with $i=2$ in \eqref{eq:ass-i}.
Some numerical experiments comparing the layer dynamics of the mass conserving Allen--Cahn and the Cahn--Hilliard equations can be found in \cite{SunWard}.

In the hyperbolic framework \eqref{eq:hyp-nonlocal}, the right hand side of the ODE \eqref{eq:ODE-hypnonlocal} is the same of \eqref{eq:ODE-nonlocalAC}, 
while in the left hand side we have two novelties: the second time derivative $\tau h''_j$ and the coefficient $\gamma_{{}_{F,g}}$ of $h'_j$.
From this point of view, we have the same results of the hyperbolic Allen--Cahn equation \eqref{eq:hypAC};
indeed, the ODEs describing the motion of the layers for \eqref{eq:hypAC} are \cite{FLM17} 
\begin{equation*}
	\tau h''_j+\gamma_{{}_{F,g}} h'_j=\frac{\e}{c_{{}_F}}\left(\alpha^{j+1}-\alpha^j\right), 
	\qquad \qquad i=1,\dots,N+1, 
\end{equation*}
and they differ from \eqref{eq:ODE-AC} only from the term $\tau h''_j$ and the coefficient $\gamma_{{}_{F,g}}$ of $h'_j$.
As we will see in Section \ref{sec:layerdyn}, the constant $\gamma_{{}_{F,g}}$ is the following \emph{weighted average} of $g$ 
\begin{equation*}
	\gamma_{{}_{F,g}}:=\frac{\displaystyle\int_{-1}^1\sqrt{F(s)}g(s)\,ds}{\displaystyle\int_{-1}^1\sqrt{F(s)}\,ds}.
\end{equation*}
In particular, when the damping coefficient is constantly equal to $1$, we have $\gamma_{{}_{F,g}}=1$,
while in the relaxation case $g(u)=1-\tau f'(u)$ one has 
\begin{align*}
	\gamma_{{}_{F,g}}&=1+\tau\int_{-1}^1\sqrt{F(s)}F''(s)\,ds\left(\int_{-1}^1\sqrt{F(s)}\,ds\right)^{-1}\\
	&=1-\tau\int_{-1}^1\frac{F'(s)^2}{2\sqrt{F(s)}}\,ds\left(\int_{-1}^1\sqrt{F(s)}\,ds\right)^{-1}<1,
\end{align*} 
where we used $f=-F'$ and integration by parts.
Hence, in the latter case the relaxation time $\tau$ appears also in the coefficient of $h'_j$, which is smaller than the constant damping case $g\equiv1$.

In general, notice that $\gamma_{{}_{F,g}}\to1$ as $g\to1$ in any reasonable way.
Reasoning as in \cite[Theorem 4.5]{FLM17}, one can compare the solutions to the systems \eqref{eq:ODE-hypnonlocal} and \eqref{eq:ODE-nonlocalAC}
and prove that if $\tau\to0^+$ and $\gamma_{{}_{F,g}}\to1$, 
then a solution to \eqref{eq:ODE-hypnonlocal} converges to the corresponding one of \eqref{eq:ODE-nonlocalAC}.

Concerning the conservation of the mass, we recall that the solution $u$ is well approximated by the function $u^{\bm h}$ satisfying \eqref{eq:uh-approx}. 
This means that $u\approx\pm1$  and denoting by $L_-$ and $L_+$ the length of all the intervals where the solution is approximately $-1$ and $+1$, respectively, we have
\begin{align*}
	L_-:&=\frac{l_1}2+\sum_{i=1}^{N/2} l_{2i+1}, \qquad \qquad \qquad  &L_+&=\sum_{i=1}^{N/2} l_{2i}+\frac{l_{N+2}}{2}, \qquad  &\mbox{ if }\, N \mbox{ is even},\\
	L_-:&=\frac{l_1}2+\sum_{i=1}^{(N-1)/2} l_{2i+1}+\frac{l_{N+2}}{2},   &L_+&=\sum_{i=1}^{(N+1)/2} l_{2i},   &\mbox{ if }\, N \mbox{ is odd},
\end{align*}
(recall that $l_j=h_j-h_{j-1}$, $\,j=2,N+1$, $\,l_1=2h_1$ and $\,l_{N+2}=2(1-h_{N+1})$).
In particular, we have that the mass of the solution is approximately given by $L_+-L_-$.
Let us compute the variation on time of the quantities $L_+$ and $L_-$.
From \eqref{eq:ODE-hypnonlocal}, we derive the following equations for the interval length $l_j=h_j-h_{j-1}$:
\begin{equation*}
	\begin{aligned}
		\tau l_1''+\gamma_{{}_{F,g}}l'_1&=\frac{2\e}{c_{{}_F}}\left(\alpha^{2}-\alpha^1+\Sigma\right),\\
		\tau l_j''+\gamma_{{}_{F,g}}l'_j&=\frac{\e}{c_{{}_F}}\left(\alpha^{j+1}-2\alpha^j+\alpha^{j-1}+2(-1)^{j+1}\Sigma\right), \qquad j=1,\dots,N+1,\\
		\tau l_{N+2}''+\gamma_{{}_{F,g}}l'_{N+2}&=-\frac{2\e}{c_{{}_F}}\left(\alpha^{N+2}-\alpha^{N+1}+(-1)^{N}\Sigma\right),
	\end{aligned}
\end{equation*}
where $\Sigma=\displaystyle\frac{1}{N+1}\sum_{i=1}^{N+1}(-1)^i(\alpha^{i+1}-\alpha^i)$.
Therefore, by using the definitions of $L_\pm$ we end up with 
\begin{equation*}
	\tau L_{\pm}''+\gamma_{{}_{F,g}}L_{\pm}'=0.
\end{equation*}
When the ODEs \eqref{eq:ODE-hypnonlocal} describe the layer dynamics of \eqref{eq:hyp-nonlocal}, 
the positions of the transition points $\bm h(0)$ and their initial velocity $\bm h'(0)$ depend on the initial data $u_0,u_1$ 
and, in particular, $\bm h'(0)$ is such that $L'_\pm(0)=0$.
Therefore, we have $L_\pm'(t)=0$ for any $t$ and this is coherent with the conservation of mass.
In general, this is different from \eqref{eq:ODE-nonlocalAC}, which directly implies $L_\pm'\equiv0$, 
while in the case of \eqref{eq:ODE-hypnonlocal}, we need a further assumption on the initial velocity of the points.

The rest of the paper is devoted to prove Theorem \ref{thm:main} and to derive the system \eqref{eq:ODE-hypnonlocal}.

\section{The coordinate system close to the submanifold $\mathcal{M}$}\label{sec:base}
The main result of this section is the smooth decomposition $u=u^{\bm h}+w$, 
where $w$ is a function of zero mean which satisfies the orthogonality condition \eqref{eq:ortho-sec2}, 
for any function $u$ sufficiently close to the constant mass submanifold $\mathcal{M}$ defined in \eqref{eq:M_0}, for details see Theorem \ref{thm:existence-coord}.
Moreover, we collect some results we use later in the proof of the main results presented in Section \ref{sec:st-main}.

\subsection{Preliminaries}
First of all, we briefly recall some properties of the $(N+1)$--manifold $\mathcal{M}^{AC}$ defined in \eqref{eq:M^AC} and
introduced by Carr and Pego in \cite{Carr-Pego}, where the authors prove that it is approximately invariant for
the Allen--Cahn equation \eqref{eq:AC}.
For any $\bm h\in\Omega_\rho$, we define the function $u^{\bm h}=u^{\bm h}(x)$, 
which approximates a metastable state with $N+1$ transition points located at $h_1,\dots,h_{N+1}$.  
To do this, we make use of the solutions to the following boundary value problem:
given $\ell>0$, let $\phi(\cdot,\ell,+1)$ be the solution to
\begin{equation}\label{eq:fi}
	\mathcal{L}(\phi):=\varepsilon^2\phi_{xx}+f(\phi)=0, \qquad \quad
	\phi\bigl(-\tfrac12\ell\bigr)=\phi\bigl(\tfrac12\ell\bigr)=0,
\end{equation}
with $\phi>0$ in $(-\tfrac12\ell,\tfrac12\ell)$,
and $\phi(\cdot,\ell,-1)$ the solution to \eqref{eq:fi} with $\phi<0$ in $(-\tfrac12\ell,\tfrac12\ell)$.
The functions $\phi(\cdot,\ell,\pm1)$ are well-defined if $\ell/\varepsilon$ is sufficiently large, and they depend on $\varepsilon$ 
and $\ell$ only through the ratio $\varepsilon/\ell$, for details see \cite{Carr-Pego} or \cite{FLM17,FLMpre}.

The function $u^{\bm h}$ is constructed by matching together the functions $\phi(\cdot,\ell,\pm1)$, using smooth cut-off functions:
given $\chi:\mathbb{R}\rightarrow[0,1]$ a $C^\infty$-function with $\chi(x)=0$ for $x\leq-1$ and $\chi(x)=1$ for $x\geq1$, set 
\begin{equation*}
	\chi^j(x):=\chi\left(\frac{x-h_j}\varepsilon\right) \qquad\textrm{and}\qquad
	\phi^j(x):=\phi\left(x-h_{j-1/2},h_j-h_{j-1},(-1)^j\right),
\end{equation*}
where 
\begin{equation*}
	h_{j+1/2}:=\tfrac12(h_j+h_{j+1})\qquad j=0,\dots,N+1,
\end{equation*}
are the middle points (note that $h_{1/2}=0$, $h_{N+3/2}=1$).
Then, we define the function $u^{\bm h}$ as
\begin{equation}\label{eq:uh}
	u^{\bm h}:=\left(1-\chi^j\right)\phi^j+\chi^j\phi^{j+1} \qquad \textrm{in}\quad I_j:=[h_{j-1/2},h_{j+1/2}],
\end{equation}
for $j=1,\dots,N+1$.
A complete list of the properties of $u^{\bm h}$ can be found in \cite{Carr-Pego};
here, we only recall that $u^{\bm h}$ is a smooth function of $\bm h$ and $x$, which satisfies \eqref{eq:uh-approx} and that
 $\mathcal{L}(u^{\bm h})=0$ except in an $\e$--neighborhood of the transition points $h_j$.
Precisely, we have
\begin{equation}\label{eq:properties-uh}
	\begin{aligned}
	u^{\bm h}(0)&=\phi(0,2h_1,-1)<0,
			&\qquad 	u^{\bm h}(h_{j+1/2})&=\phi\left(0,h_{j+1}-h_j,(-1)^{j+1}\right),\\
	u^{\bm h}(h_j)&=0,
			&\qquad \mathcal{L}(u^{\bm h}(x))&=0\quad \textrm{for }|x-h_j|\geq\varepsilon,
	\end{aligned}
\end{equation}
for any $j=1,\dots,N+1$.

Now, we give the precise definition of the quantities $\alpha^j$ introduced in Section \ref{sec:st-main} 
and appearing in the ODEs \eqref{eq:ODE-hypnonlocal}, \eqref{eq:ODE-AC} and \eqref{eq:ODE-nonlocalAC}.
Since $\phi(0,\ell,\pm1)$ depends only on the ratio $r=\varepsilon/\ell$, we can define
\begin{equation*}
	\alpha_\pm(r):=F(\phi(0,\ell,\pm1)), \qquad \quad \beta_\pm(r):=1\mp\phi(0,\ell,\pm1),
\end{equation*}
where we recall that $f=-F'$.
By definition, $\phi(0,\ell,\pm1)$ is close to $+1$ or $-1$ and so, $\alpha_\pm(r), \beta_\pm(r)$ are close to $0$. 
The next result characterizes the leading terms in $\alpha_\pm$ and $\beta_\pm$ as $r\to 0$.
\begin{prop} [Carr--Pego \cite{Carr-Pego}] \label{prop:alfa,beta}
Let $F$ be such that \eqref{eq:ass-F} holds and set 
\begin{equation*}
	A_\pm^2:=F''(\pm1), \qquad K_{\pm}=2\exp\left\{\int_0^1\left(\frac{A_\pm}{(2F(\pm t))^{1/2}}-\frac{1}{1-t}\right)\,dt\right\}.
\end{equation*}
There exists $r_0>0$ such that if $0<r<r_0$, then
\begin{equation*}
	\begin{aligned}
	\alpha_\pm(r)&=\tfrac12K^2_\pm A^2_\pm\,\exp(-{A_\pm}/r\bigr)\bigl\{1+O\left(r^{-1} \exp(-{A_\pm}/2r)\right)\bigr\},\\
	\beta_\pm(r)&=K_\pm\,\exp\bigl(-{A_\pm}/2r\bigr)\bigl\{1+O\left(r^{-1} \exp(-{A_\pm}/2r)\right)\bigr\},
	\end{aligned}
\end{equation*}
with corresponding asymptotic formulae for the derivatives of $\alpha_\pm$ and $\beta_\pm$.
\end{prop}
For $j=1,\dots,N+1$, we set
\begin{equation*}
	l_j:=h_{j+1}-h_{j}, \qquad \qquad r_{j}:=\frac{\varepsilon}{l_j},
\end{equation*}
and
\begin{equation*}
	\alpha^{j}:=\left\{\begin{aligned}
		&\alpha_+(r_{j}) 	&j \textrm{ odd},\\
		&\alpha_-(r_{j})  	&j \textrm{ even},\\
		\end{aligned}\right.
	\qquad
	\beta^{j}:=\left\{\begin{aligned}
		&\beta_+(r_{j})	&j \textrm{ odd},\\
		&\beta_-(r_{j})	&j \textrm{ even}.\\
		\end{aligned}\right.
\end{equation*}

Finally, let us introduce the \emph{barrier function}
\begin{equation}\label{eq:barrier}
	\Psi(\bm h):=\sum_{j=1}^{N+1}{\langle\mathcal{L}\bigl(u^{\bm h}\bigr),k^{\bm h}_j\rangle}^2=\sum_{j=1}^{N+1}\bigl(\alpha^{j+1}-\alpha^{j}\bigr)^2,
\end{equation}
where $\mathcal{L}$ is the Allen--Cahn differential operator introduced above and the functions $k^{\bm h}_j$ are defined by
\begin{equation*}
	k^{\bm h}_j(x):=-\gamma^j(x)u^{\bm h}_x(x), \qquad  \mbox{ with } \; 
	\gamma^j(x):=\chi\left(\frac{x-h_{j}-\varepsilon}\varepsilon\right)\left[1-\chi\left(\frac{x-h_{j+1}+\varepsilon}\varepsilon\right)\right].
\end{equation*}
By construction, $k^{\bm h}_j$ are smooth functions of $x$ and $\bm h$ and are such that
\begin{equation*}
	\begin{aligned}
	k^{\bm h}_j(x)&=0				&\quad \textrm{for}\quad &x\notin[h_{j-1/2},h_{j+1/2}],\\
	k^{\bm h}_j(x)&=-u^{\bm h}_x(x)	&\quad \textrm{for}\quad &x\in[h_{j-1/2}+2\varepsilon,h_{j+1/2}-2\varepsilon]. 
	\end{aligned}
\end{equation*}
As the function $u^{\bm h}$, the functions $k^{\bm h}_j(x)$ and $\Psi(\bm h)$ are introduced in \cite{Carr-Pego}:
$k^{\bm h}_j$ are approximate tangent vectors to the manifold $\mathcal{M}^{AC}$ and 
the barrier function $\Psi(\bm h)$ may be considered an approximation of the quantity $\|P^{\bm h}\mathcal{L}\bigl(u^{\bm h}\bigr)\|^2$,
where $P^{\bm h}$ is the projection to the tangent space to $\mathcal{M}^{AC}$ at $u^{\bm h}$.

\subsection{The constant mass submanifold}
As it was previously mentioned, we use different variables to describe a function $u$ 
sufficiently close to the constant mass submanifold manifold $\mathcal{M}$ defined in \eqref{eq:M_0}.
First of all, we recall that the manifold $\mathcal{M}$ can be parametrized by the first $N$ components of the vector $\bm h$ 
and the component $h_{N+1}$ can be seen as a function of $h_1,\dots,h_{N+1}$; 
precisely, if $u^{\bm h}\in\mathcal{M}$, then we have $h_{N+1}=z(h_1,\dots,h_N)$ for some $z:\R^N\to\R$ satisfying
\begin{equation}\label{eq:derh_N+1}
	z_j:=\frac{\partial z}{\partial h_j}=(-1)^{N-j}+\mathcal{O}\left(\e^{-1}\exp(-A\ell^{\bm h}/\e)\right),
\end{equation}
where $A:=\sqrt{\min\{F''(-1),F''(1)\}}$ and $\ell^{\bm h}:=\min\{h_j-h_{j-1}\}$ as in Theorem \ref{thm:main};
for details see \cite[Lemma 2.4]{FLMpre}.
Therefore, in what follows we denote by $\bm\xi\in\R^N$ the vector of the first $N$ components of $\bm h$ 
and we interchangeably use $\bm\xi$ and $\bm h$, meaning that $\bm h=(\bm\xi,z(\bm\xi))$.
In particular, we use the notations $u^{\bm\xi}$ for $u^{(\bm\xi,z(\bm\xi))}$ and
\begin{equation}\label{eq:uj}
	u^{\bm\xi}_j:=\frac{\partial u^{\bm\xi}}{\partial\xi_j}=u^{\bm h}_j+z_j\,u^{\bm h}_{N+1}, 
	\qquad\quad \mbox{where } \quad u^{\bm h}_j:=\frac{\partial u^{\bm h}}{\partial h_j}.
\end{equation}
Accordingly to the new variables, we define the functions $\nu^{\bm h}_j$ as 
\begin{equation}\label{eq:newtangvec}
	\nu^{\bm h}_j:=k^{\bm h}_j+(-1)^{N-j}k^{\bm h}_{N+1}, \qquad \qquad j=1,\dots N,
\end{equation}
where $k^{\bm h}_j$ are the approximations of the tangent vectors to $\mathcal{M}^{AC}$ introduced above.
Similarly as $u^{\bm\xi}$, we shall use the notation $\nu^{\bm\xi}_j$ for $\nu^{(\bm\xi,z(\bm\xi))}_j$ and
\begin{equation}\label{eq:nuji}
	\nu^{\bm\xi}_{ji}:=\frac{\partial \nu^{\bm\xi}_j}{\partial\xi_i}=k^{\bm h}_{ji}+z_jk^{\bm h}_{j,N+1}+(-1)^{N-j}k^{\bm h}_{N+1,i}+(-1)^{N-j}z_jk^{\bm h}_{N+1,N+1}.
\end{equation}
The idea of using the functions $\nu^{\bm h}_j$ \eqref{eq:newtangvec} instead of $k^{\bm h}_j$ is crucial in our study
and it can be also applied to the study of the mass conserving Allen--Cahn equation \eqref{eq:maco-AC}.

In the following proposition we collect some estimates concerning $u_j^{\bm\xi}$, $\nu^{\bm\xi}_j$ and their derivatives, which will be useful in the sequel.
\begin{prop}\label{prop:est-u-nu}
Fix $F\in C^3(\mathbb{R})$ satisfying \eqref{eq:ass-F} and define $c_{{}_F}$ as in \eqref{eq:c_F}.
Given $N\in\mathbb{N}$ and $\delta\in(0,1/N+1)$,
there exist positive constants $\e_0,C$ such that if $\e$ and $\rho$ satisfy \eqref{eq:triangle} and $\bm h=(\bm\xi,z(\bm\xi))\in\Omega_\rho$, then
\begin{align}
	\e\|u^{\bm\xi}_j\|_{{}_{L^\infty}}+\e^{1/2}\|u^{\bm\xi}_j\|+\e^{1/2}\|\nu^{\bm\xi}_j\|&\leq C,\label{eq:uj-est}\\
	\langle u^{\bm\xi}_j,\nu^{\bm\xi}_j\rangle&=2c_{{}_F}\,\e^{-1}+\mathcal{O}\left(\exp(-C/\e)\right), \label{eq:uj,nuj}\\
	\int_0^1\nu^{\bm\xi}_j\,dx&=\mathcal{O}\left(\exp(-C/\e)\right),\label{eq:int-nu}
\end{align}	
for $j=1,\dots,N$.
Moreover, if $i\neq j$, we have
\begin{equation}\label{eq:uj,nui}
	\langle u^{\bm\xi}_i,\nu^{\bm\xi}_j\rangle=(-1)^{i+j}c_{{}_F}\,\e^{-1}+\mathcal{O}\left(\exp(-C/\e)\right).
\end{equation}
Finally,  
\begin{equation}\label{eq:nuij-est}
	\e^{3/2}\|\nu^{\bm\xi}_{ij}\|+\e\|\nu^{\bm\xi}_{ij}\|_{{}_{L^1}}\leq C, 
\end{equation}
for any $i,j=1,\dots,N+1$.
\end{prop}
\begin{proof}
The proof of the estimates \eqref{eq:uj-est}-\eqref{eq:nuij-est} follows from the definitions of $u_j^{\bm\xi}$, $\nu_j^{\bm\xi}$ \eqref{eq:uj}, \eqref{eq:newtangvec} and the properties of the  functions $u_j^{\bm h}$, $k_j^{\bm h}$ proved in \cite{Carr-Pego}.  	
	
The estimates \eqref{eq:uj-est} are a consequence of the definitions \eqref{eq:uj}-\eqref{eq:newtangvec}, 
the formula \eqref{eq:derh_N+1}, the fact that $\e,\rho$ satisfy \eqref{eq:triangle} and \cite[Proposition 2.3-Lemma 8.3]{Carr-Pego}.

Similarly, one can obtain the estimates \eqref{eq:nuij-est},
which follow from
\begin{equation*}
	\e^{3/2}\|k^{\bm h}_{ij}\|+\e\|k^{\bm h}_{ij}\|_{{}_{L^1}}\leq C,
\end{equation*}
for any $i,j=1,\dots,N+1$, and the definition \eqref{eq:nuji}.

Moreover, by using the definitions \eqref{eq:uj}-\eqref{eq:newtangvec}, we infer
\begin{equation*}
	\langle u^{\bm\xi}_i,\nu^{\bm\xi}_j\rangle=\langle u^{\bm h}_i,k^{\bm h}_j\rangle+(-1)^{N-j}\langle u^{\bm h}_i,k^{\bm h}_{N+1}\rangle
	+z_i\langle u^{\bm h}_{N+1},k^{\bm h}_j\rangle+(-1)^{N-j}z_i\langle u^{\bm h}_{N+1},k^{\bm h}_{N+1}\rangle,
\end{equation*}
for $i,j=1,\dots,N$.
Since, for \cite[Theorem 3.5]{Carr-Pego}, one has
\begin{equation}\label{eq:matrix-AC}
	\begin{aligned}
	\langle u^{\bm h}_i,k^{\bm h}_j\rangle&=\mathcal{O}\left(\exp(-C/\e)\right), \qquad \qquad\qquad & i\neq j,  \\
	\langle u^{\bm h}_i,k^{\bm h}_i\rangle&=c_{{}_F}\,\e^{-1}+\mathcal{O}\left(\exp(-C/\e)\right), & i=1,\dots, N+1,	
	\end{aligned}
\end{equation}
by using again \eqref{eq:triangle}-\eqref{eq:derh_N+1}, we obtain \eqref{eq:uj,nuj} and \eqref{eq:uj,nui}.

It remains to prove \eqref{eq:int-nu}.
By definition, we get
\begin{equation*}
	\int_0^1\nu^{\bm\xi}_j\,dx=\int_0^1k^{\bm h}_j\,dx+(-1)^{N-j}\int_0^1k^{\bm h}_{N+1}\,dx,  \qquad\qquad j=1,\dots,N,
\end{equation*} 
and 
\begin{align*}
	\int_0^1k^{\bm h}_j\,dx&=-\int_{I_j}u_x^{\bm h}\,dx+\int_{I_j}(1-\gamma^j)u_x^{\bm h}\,dx\\
	&=u^{\bm h}(h_{j-1/2})-u^{\bm h}(h_{j+1/2})+\mathcal{O}\left(\exp(-C/\e)\right), \qquad\qquad j=1,\dots,N+1,
\end{align*}
where we used the estimate \cite[Eq. (4.13)]{FLM17}.
Using \eqref{eq:properties-uh}, the definition of $\beta_j$ and Proposition \ref{prop:alfa,beta}, we deduce
\begin{equation*}
	u^{\bm h}(h_{j-1/2})=(-1)^{j}+(-1)^{j+1}\beta^{j+1}=(-1)^{j}+\mathcal{O}\left(\exp(-C/\e)\right), \qquad j=1,\dots,N+2,
\end{equation*}
and, as a consequence 
\begin{equation*}
	\int_0^1k^{\bm h}_j\,dx=2(-1)^j+\mathcal{O}\left(\exp(-C/\e)\right), \qquad\qquad j=1,\dots,N+1.
\end{equation*}
Therefore, we end up with
\begin{equation*}
	\int_0^1\nu^{\bm\xi}_j\,dx=2(-1)^j+2(-1)^{2N-j+1}+\mathcal{O}\left(\exp(-C/\e)\right)\qquad\qquad j=1,\dots,N,
\end{equation*} 
and the proof is complete.
\end{proof}

Let $S(\bm\xi)$ be the $N\times N$ matrix with elements $s_{ij}(\bm\xi):=\langle u^{\bm\xi}_j,\nu^{\bm\xi}_i\rangle$;
from Proposition \ref{prop:est-u-nu} it follows that
\begin{equation}\label{eq:S-matrix}
	S(\bm\xi):=\frac{c_{{}_F}}{\e}\left(\begin{array}{ccccc}  2 & -1 & 1 & \dots & (-1)^{N+1}\\
	-1 & 2 & -1 & \dots & (-1)^{N}\\
	1 & -1 & 2 & \dots & (-1)^{N+1}\\
	\dots & \dots & \dots & \dots & \dots\\
	(-1)^{N+1}& (-1)^{N} & (-1)^{N+1} & \dots & 2
	\end{array}\right)+\mathcal{O}\left(\exp(-C/\e)\right),
\end{equation}
where we recall $c_{{}_F}$ is defined in \eqref{eq:c_F}.
By inverting such matrix, we obtain
\begin{equation}\label{eq:S^-1}
	S^{-1}(\bm\xi):=\frac{\e}{(N+1)c_{{}_F}}\left(\begin{array}{ccccc}  N & 1 & -1 & \dots & (-1)^{N}\\
	1 & N & 1 & \dots & (-1)^{N+1}\\
	-1 & 1 & N & \dots & (-1)^{N}\\
	\dots & \dots & \dots & \dots & \dots\\
	(-1)^{N}& (-1)^{N+1} & (-1)^{N} & \dots & N
	\end{array}\right)+\mathcal{O}\left(\exp(-C/\e)\right).
\end{equation}
The matrix $S$ is different with respect to the cases of the Allen--Cahn equation, where it is (up to the small error) a diagonal matrix for \eqref{eq:matrix-AC},
and to the Cahn--Hilliard equation, where it is (up to a small error) a lower triangular matrix, see \cite[pag. 18]{FLMpre}.
We will use later the formula for $S^{-1}$ to determine the system of ODEs which describes the movement of the layers.

Now, we have all the tools to prove the existence of the smooth decomposition 
$u=u^{\bm h}+w$ with $u^{\bm h}\in\mathcal{M}$ and $w$ satisfying 
\begin{equation}\label{eq:cond-w}
	\int_0^1 w\,dx=0, \qquad w_x(0)=w_x(1)=0, \qquad  \langle w,\nu^{\bm\xi}_j\rangle=0, \;\; j=1,\dots,N,
\end{equation}
for any function $u$ in a small neighborhood of $\mathcal{M}$.
We emphasize that in \cite{Carr-Pego} the authors prove the existence of the coordinates $(\bm h,w)$, with $w$ orthogonal to $k^{\bm h}_j$ and $u^{\bm h}\in\mathcal{M}^{AC}$,
while in our work, we need $u^{\bm h}\in\mathcal{M}$, i.e. $u^{\bm h}$ with mass equal to $M$;
hence, we need a further condition on $w$, that is $\displaystyle\int_0^1 w\,dx=0$.
\begin{thm}\label{thm:existence-coord}
Given $N\in\mathbb{N}$ and $\delta\in(0,1/N+1)$,
there exists $\e_0>0$ such that if $\e$, $\rho$ satisfy \eqref{eq:triangle} and $u$ satisfies 
\begin{equation}\label{eq:ass-excoord}
	 \int_0^1u\,dx=M, \qquad u_x(0)=u_x(1)=0, \qquad  \mbox{ and } \qquad \|u-u^{\bm h}\|_{{}_{L^\infty}}\leq\e^2,
\end{equation}
for some $\bm h\in\Omega_\rho$, then there is a unique $\bar{\bm{h}}\in\Omega_\rho$ such that $u=u^{\bar{\bm h}}+w$ with $w$ satisfying
\begin{equation}\label{eq:w-excoord}
	\int_0^1 w\,dx=0, \qquad w_x(0)=w_x(1)=0, \qquad  \langle w,\nu^{\bar{\bm h}}_j\rangle=0, \;\; j=1,\dots,N,
\end{equation}
where the functions $\nu_j^{\bm h}$ are defined in \eqref{eq:newtangvec}.
 
Moreover, if $\|u-u^{\bm{h}^*}\|_{{}_{L^\infty}}=\inf\{ \|u-u^{\bm h}\|_{{}_{L^\infty}}\, :\, \bm h\in\Omega_\rho\}$ for some $\bm{h}^*\in\Omega_\rho$,
then there exists a positive constant $C$ such that
\begin{equation}\label{eq:h-h*}
	|\bar{\bm{h}}-\bm{h}^*|\leq C\e\|u-u^{\bm{h}^*}\|_{{}_{L^\infty}}, \qquad \mbox{ and } \qquad \|u-u^{\bar{\bm{h}}}\|_{{}_{L^\infty}}\leq C\|u-u^{\bm{h}^*}\|_{{}_{L^\infty}}.
\end{equation}
\end{thm}
\begin{proof}
We proceed as in \cite[Section 9]{Carr-Pego} and \cite[Theorem A.7]{Bates-Xun2}.
For any $u$ satisfying \eqref{eq:ass-excoord}, the existence of the decomposition $u=u^{\bm h}+w$ with $w$ satisfying \eqref{eq:cond-w} 
is equivalent to the existence of $\bm\xi\in\R^N$ such that $(\xi_1,\dots,\xi_N,z(\xi_1,\dots,\xi_N))=\bm h\in\Omega_\rho$ and
\begin{equation*}
	\Theta_j(\bm\xi,u):=\langle u-u^{\bm\xi},\nu_j^{\bm\xi}\rangle=0, \qquad \qquad \mbox{for any }\, j=1,\dots,N.
\end{equation*}
To this aim, we define the functions $\bm\Theta(\bm\xi,u)=(\Theta_1(\bm\xi,u),\dots,\Theta_N(\bm\xi,u))$ and
\begin{equation*}
	\bm\Lambda(\bm\xi,u,\bm\xi^*):=\bm\xi+S^{-1}(\bm\xi^*)\bm\Theta(\bm\xi,u),
\end{equation*}
where $\bm{\xi^*}$ is the vector of the first $N$ components of $\bm h^*$ and $\bm h^*$ is the same of \eqref{eq:h-h*}.
Precisely, our goal is to prove that there exists a unique fixed point of $\bm\Lambda(\cdot,u,\bm\xi^*)$ in a neighborhood of $\bm\xi^*$. 
We claim that if $\e_0$ is sufficiently small and $|\bm\xi-\bm\xi^*|\leq\e^2$, then $\|\partial\bm\Lambda/\partial\bm\xi\|_{{}_\infty}\leq\frac14$, 
where $\|\cdot\|_{{}_\infty}$ is the matrix norm induced by the vector norm $|\cdot|_{{}_\infty}$.
Indeed, since
\begin{equation*}
	\frac{\partial\bm\Lambda}{\partial\bm\xi}=\mathbb{I}_N+S^{-1}(\bm\xi^*)\frac{\partial\bm\Theta}{\partial\bm\xi}
	=S^{-1}(\bm\xi^*)\left[S(\bm\xi^*)-S(\bm\xi)+B(\bm\xi)\right], 
\end{equation*}
where $B_{ij}=\langle u-u^{\bm\xi},\nu_{ij}^{\bm\xi}\rangle$, we deduce
\begin{align*}
	\left\|\frac{\partial\bm\Lambda}{\partial\bm\xi}\right\|_{{}_\infty}&\leq\left\|S^{-1}(\bm\xi^*)\right\|_{{}_\infty}
	\left\|S(\bm\xi^*)-S(\bm\xi)+B(\bm\xi)\right\|_{{}_\infty}\\
	&\leq C\e\left(\left\|S(\bm\xi^*)-S(\bm\xi)\right\|_{{}_\infty}+\e^{-1}\|u-u^{\bm\xi}\|_{{}_{L^\infty}}\right),
\end{align*}
where we used the estimates $\left\|S^{-1}(\bm\xi^*)\right\|_{{}_\infty}\leq C\e$ and $\|\nu_{ij}^{\bm\xi}\|_{{}_{L^1}}\leq C\e^{-1}$ 
for some $C>0$ independent on $\e$ (see \eqref{eq:S^-1} and \eqref{eq:nuij-est}).
Let us estimate the last term as follows
\begin{equation*}
	\|u-u^{\bm\xi}\|_{{}_{L^\infty}}\leq\|u-u^{\bm\xi^*}\|_{{}_{L^\infty}}+\|u^{\bm\xi^*}-u^{\bm\xi}\|_{{}_{L^\infty}}\leq\e^2+C\e^{-1}|\bm\xi-\bm\xi^*|,
\end{equation*}
where we used the definition of $\bm\xi^*$, \eqref{eq:uj-est} and the assumption \eqref{eq:ass-excoord}.
By using the latter estimate and \eqref{eq:S-matrix}, we conclude that if $|\bm\xi-\bm\xi^*|\leq\e^2$, then
\begin{equation*}
	\left\|\frac{\partial\bm\Lambda}{\partial\bm\xi}\right\|_{{}_\infty}=\mathcal{O}(\e), \qquad \mbox{ as }\, \e\to0^+.
\end{equation*}
Hence, we can choose $\e_0$ so small that $\|\partial\bm\Lambda/\partial\bm\xi\|_{{}_\infty}\leq\frac14$.
Such estimates allows us to prove that $\bm\Lambda(\cdot,u,\bm\xi^*)$ is a contraction if $|\bm\xi-\bm\xi^*|\leq\e^2$.
Indeed, we have
\begin{align}
	|\bm\Lambda(\bm\xi,u,\bm\xi^*)-\bm\xi^*|_{{}_\infty}&\leq |\bm\Lambda(\bm\xi,u,\bm\xi^*)-\bm\Lambda(\bm\xi,u^{\bm\xi^*},\bm\xi^*)|_{{}_\infty}+
	 |\bm\Lambda(\bm\xi,u^{\bm\xi^*},\bm\xi^*)-\bm\Lambda(\bm\xi^*,u^{\bm\xi^*},\bm\xi^*)|_{{}_\infty}\notag\\
	&\leq\|S^{-1}(\bm\xi^*)\|_{{}_\infty} |\bm\Theta(\bm\xi,u)-\bm\Theta(\bm\xi,u^{\bm\xi^*})|_{{}_\infty} +\frac14|\bm\xi-\bm\xi^*|_{{}_\infty}\notag\\
	&\leq C\e\|u-u^{\bm\xi^*}\|_{{}_{L^\infty}}+\frac14\e^2<\frac{\e^2}{2},\label{eq:Lambda-first}
\end{align}
provided $C\e<\frac14$.
Therefore, $\bm\Lambda(\cdot,u,\bm\xi^*)$ is a contraction in $|\bm\xi-\bm\xi^*|\leq\e^2$ and, as a consequence, has a unique fixed point $\bar{\bm\xi}$.
It follows that there exists $\bar{\bm h}\in\Omega_\rho$ such that $u=u^{\bar{\bm h}}+w$ with $w$ satisfying \eqref{eq:w-excoord}.
Next, we show that such representation is unique and we prove \eqref{eq:h-h*}.

To prove the uniqueness of the tubular coordinates we use \cite[Lemma 9.2]{Carr-Pego} and 
the fact that if $\bm h^*$ and $\bm h^{**}$ belong to $\Omega_\rho$ with $\rho$ sufficiently small, then 
\begin{equation}\label{eq:h*-h**}
	 \|u^{\bm h^*}-u^{\bm h^{**}}\|_{{}_{L^\infty}}<2\e^2 \qquad \quad \Longrightarrow \qquad \quad 
	 |\bm h^*-\bm h^{**}|_{{}_\infty}<\frac{\e^2}{2}.
\end{equation}
Let us assume that there exists $\bm h^{**}\in\Omega_\rho$, such that $u=u^{\bm h^{**}}+w^{**}$ with 
$\|w^{**}\|_{{}_{L^\infty}}<\e^2$ and $\langle w^{**},\nu^{\bm\xi^{**}}_j\rangle=0$, for $j=1,\dots,N$.
Then, we infer
\begin{equation*}
	\|u^{\bm h^*}-u^{\bm h^{**}}\|_{{}_{L^\infty}}\leq\|u^{\bm h^*}-u\|_{{}_{L^\infty}}+\|u-u^{\bm h^{**}}\|_{{}_{L^\infty}}<2\e^2.
\end{equation*}
Hence, by using \eqref{eq:h*-h**} we obtain $|\bm\xi^{**}-\bm\xi^*|<\e^2$, and so $\bm\xi^{**}=\bar{\bm\xi}$, which implies $\bm h^{**}=\bar{\bm h}$.

Now, we prove the first inequality of \eqref{eq:h-h*};
by reasoning as in \eqref{eq:Lambda-first} we get
\begin{equation*}
	|\bar{\bm\xi}-\bm\xi^*|_{{}_\infty}=|\bm\Lambda(\bar{\bm\xi},u,\bm\xi^*)-\bm\xi^*|_{{}_\infty}\leq C\e\|u-u^{\bm\xi^*}\|_{{}_{L^\infty}}+\frac14|\bar{\bm\xi}-\bm\xi^*|_{{}_\infty}.
\end{equation*}
Thus, by using \eqref{eq:derh_N+1} we obtain the first inequality of \eqref{eq:h-h*}.
Concerning the second one, we have 
\begin{align*}
	\|u-u^{\bar{\bm{h}}}\|_{{}_{L^\infty}}&=\|u-u^{\bar{\bm{\xi}}}\|_{{}_{L^\infty}}\leq\|u-u^{\bm\xi^*}\|_{{}_{L^\infty}}+\|u^{\bm\xi^*}-u^{\bar{\bm{\xi}}}\|_{{}_{L^\infty}}\\
	&\leq\|u-u^{\bm\xi^*}\|_{{}_{L^\infty}}+C\e^{-1}|\bar{\bm\xi}-\bm\xi^*|_{{}_\infty}\leq C\|u-u^{\bm\xi^*}\|_{{}_{L^\infty}},
\end{align*}
and the proof is complete.
\end{proof}

\section{Existence of the slow channel}\label{sec:slow}
This section is devoted to the proof of Theorem \ref{thm:main}.
Thanks to Theorem \ref{thm:existence-coord}, we can use the decomposition $u=u^{\bm\xi}+w$ with $w$ satisfying \eqref{eq:cond-w}, 
to study the dynamics of the solutions in a tubular neighborhood of the manifold $\mathcal{M}_{{}_0}$. 
Let us rewrite equation \eqref{eq:hyp-nonlocal} as
\begin{equation}\label{eq:system-u-v}
	\begin{cases}
		u_t=v,\\
		\displaystyle\tau v_t=\mathcal{L}(u)-g(u)v-\int_0^1 f(u)\,dx-\int_0^1[1-g(u)]v\,dx,
	\end{cases}
\end{equation}
where $\mathcal{L}(u):=\e^2u_{xx}+f(u)$.
From system \eqref{eq:system-u-v} and the decomposition $u=u^{\bm\xi}+w$, it follows that
\begin{equation*}
	\begin{cases}
		\displaystyle w_t=v-\sum_{j=1}^{N}u^{\bm\xi}_j\xi_j',\\
		\displaystyle\tau v_t=\mathcal{L}(u^{\bm\xi}+w)-g(u^{\bm\xi}+w)v-\int_0^1 f(u^{\bm\xi}+w)\,dx-\int_0^1[1-g(u^{\bm\xi}+w)]v\,dx.
	\end{cases}
\end{equation*}
By using the expansion 
\begin{equation*}
	\mathcal{L}(u^{\bm\xi}+w)=\mathcal{L}(u^{\bm\xi})-L^{\bm\xi}w -f_2w^2, \qquad \mbox{ where } \qquad f_2:=\int_0^1 (1-s)f''(u^{\bm\xi}+sw)\,ds,
\end{equation*}
and $L^{\bm\xi}$ is the linearized operator of $\mathcal{L}$ about $u^{\bm\xi}$, that is $L^{\bm\xi}w:=-\e^2w_{xx}-f'(u^{\bm\xi})w$,
we rewrite the system for $(w,v)$ in the form
\begin{equation}\label{eq:system-w-v}
	\begin{cases}
		\displaystyle w_t=v-\sum_{j=1}^{N}u^{\bm\xi}_j\xi_j',\\
		\displaystyle\tau v_t=\mathcal{L}(u^{\bm\xi})-L^{\bm\xi}w -f_2w^2-g(u^{\bm\xi}+w)v-\int_0^1 f(u^{\bm\xi}+w)\,dx-\int_0^1[1-g(u^{\bm\xi}+w)]v\,dx.
\end{cases}
\end{equation}
Now, we derive the equation for $\bm\xi$ by differentiating with respect to $t$ the orthogonality condition in \eqref{eq:cond-w}:
\begin{equation}\label{eq:xi-ort}
	\sum_{i=1}^{N}\left\{\langle u^{\bm\xi}_i,\nu^{\bm\xi}_j\rangle-\langle w,\nu^{\bm\xi}_{ji}\rangle\right\}\xi_i'=\langle v,\nu^{\bm\xi}_j\rangle, \qquad \qquad j=1,\dots,N,
\end{equation}
which can be rewritten in the compact form
\begin{equation}\label{eq:xi-compact}
	\hat{S}(\bm\xi,w)\bm\xi'=\bm Y(\bm\xi,v),
\end{equation}
where 
\begin{equation*}
	\hat{S}_{ji}(\bm\xi,w):=\langle u^{\bm\xi}_i,\nu^{\bm\xi}_j\rangle-\langle w,\nu^{\bm\xi}_{ji}\rangle, \qquad \qquad Y_j(\bm\xi,v):=\langle v,\nu^{\bm\xi}_j\rangle.
\end{equation*}
Combining \eqref{eq:system-w-v} and \eqref{eq:xi-compact}, we end up with the ODE-PDE coupled system
\begin{equation}\label{eq:system-w-v-xi}
	\begin{cases}
		\displaystyle w_t=v-\sum_{j=1}^{N}u^{\bm\xi}_j\xi_j',\\
		\displaystyle\tau v_t=\mathcal{L}(u^{\bm\xi})-L^{\bm\xi}w -f_2w^2-g(u^{\bm\xi}+w)v-\int_0^1 f(u^{\bm\xi}+w)\,dx-\int_0^1[1-g(u^{\bm\xi}+w)]v\,dx,\\
		\hat{S}(\bm\xi,w)\bm\xi'=\bm Y(\bm\xi,v).
\end{cases}
\end{equation}
The next step is to study the dynamics of the solutions to \eqref{eq:system-w-v-xi} when $(w,v,\bm\xi)$ satisfying appropriate assumptions. 
Precisely, we define the spaces
\begin{equation*}
	W:=\left\{w\in H^2(0,1)\, : \, w\, \textrm{ satisfies \eqref{eq:cond-w}}\right\}, \qquad \qquad V:=\left\{v\in L^2(0,1)\, : \, \int_0^1 v\,dx=0\right\},
\end{equation*}
the functional
\begin{equation}\label{eq:functional-Exi}
	E^{\bm\xi}[w,v]:=\frac12\langle w,L^{\bm\xi}w\rangle+\frac12\tau\|v\|^2+\e\tau\langle w,v\rangle,
\end{equation}
and for $\Gamma>0$, the set
\begin{align*}
	\hat{\mathcal{Z}}_{{}_{\Gamma,\rho}}:=\Bigg\{(w,v,\bm\xi)\, :\, (w,v)\in W\times V, \quad \bm\xi \, \mbox{ is such that }\, &
	\bm h=(\bm\xi,z(\bm\xi))\in\bar\Omega_\rho\\ 
	& \;\qquad\mbox{ and } \,  E^{\bm\xi}[w,v]\leq\Gamma\Psi(\bm h)\Bigg\},
\end{align*}
where the barrier function $\Psi$ is defined in \eqref{eq:barrier}.
It is well known that the linearized operator $L^{\bm\xi}$ has an infinite sequence of real and simple eigenvalues satisfying 
$\lambda_1<\lambda_2<\dots<\lambda_j\to+\infty$ as $j\to+\infty$ and that the first $N+1$ eigenvalues are exponentially small as $\e\to0$, namely
\begin{equation*}
	\max_{1\leq j\leq N+1}|\lambda_j|\leq C\exp(-A\ell^{\bm h}/2\e), \qquad \qquad \lambda_{N+2}>\Lambda_0,
\end{equation*} 
for some $C,\Lambda_0>0$ independent on $\e$; for details see \cite[Section 4]{Carr-Pego}.
Moreover, Carr and Pego \cite{Carr-Pego} proved that $L^{\bm\xi}$ is coercive in directions not tangent to the manifold $\mathcal{M}^{AC}$; 
precisely, they show that if $w\in H^2(0,1)$ satisfies $w_x=0$ at $x=0,1$, and 
\begin{equation}\label{eq:condforcoerc}
	\sup\left\{\frac{\langle w,\kappa\rangle}{\|w\|\|\kappa\|}\, :\, \kappa\in\mbox{span}\left\{k_1^{\bm h},\dots,k_{N+1}^{\bm h}\right\}\right\}\leq\cos\theta_0, 	
\end{equation}
for some $\theta_0\in(0,\frac{\pi}{2}]$, then there exists $\Lambda_0>0$ (independent on $\e$) such that
\begin{equation}\label{eq:coercive}
	\frac12\Lambda_0\e\|w\|^2_{{}_{L^\infty}}\leq\Lambda_0\int_0^1(\e^2 w_x^2+w^2)\,dx\leq\langle w,L^{\bm\xi}w\rangle.
\end{equation} 
In the case of the Allen--Cahn equation \eqref{eq:AC}, 
the reminder function $w$ is orthogonal to $k^{\bm h}_j$ for any $j$ and as a trivial consequence, the condition \eqref{eq:condforcoerc} is satisfied.
In our case, $w$ satisfies \eqref{eq:cond-w} and so, we assume that $\langle w,\nu^{\bm\xi}_j\rangle=0$, for $j=1,\dots,N$.
The latter property does not imply the orthogonality to $k^{\bm h}_j$ for $j=1,\dots,N+1$ and before using the property \eqref{eq:coercive}, 
we need to prove that $w\in W$ implies $w$ satisfying condition \eqref{eq:condforcoerc}.
\begin{lem}
Given $N\in\mathbb{N}$ and $\delta\in(0,1/N+1)$,
there exists $\e_0>0$ such that if $\e,\rho$ satisfy \eqref{eq:triangle}, $\bm h\in\Omega_\rho$ and $w\in W$, then $w$ satisfies \eqref{eq:coercive}.  
\end{lem}
\begin{proof}
Fix $\bm h\in\Omega_\rho$ and $w\in W$. 
Let us prove that $w\notin\mbox{span}\left\{k_1^{\bm h},\dots,k_{N+1}^{\bm h}\right\}$.
By contradiction, assume $w=\displaystyle\sum_{i=i}^{N+1}c_i k_i^{\bm h}$, for some $c_i\in\R$, satisfying $\displaystyle\sum_{i=1}^{N+1}|c_i|>0$.
Hence, 
\begin{equation*}
	\langle w,\nu_j^{\bm\xi}\rangle=\sum_{i=1}^{N+1}c_i\langle k_i^{\bm h},\nu_j^{\bm\xi}\rangle=
	c_j\|k_j^{\bm h}\|^2+(-1)^{N-j}c_{N+1}\|k_{N+1}^{\bm h}\|^2, \qquad \qquad j=1,\dots,N,
\end{equation*}
where we used definition \eqref{eq:newtangvec} and the fact that $\langle k_i^{\bm h},k_j^{\bm h}\rangle=0$ for $i\neq j$.
If $c_{N+1}=0$, the condition $\langle w,\nu_j^{\bm\xi}\rangle=0$ for any $j=1,\dots,N$ implies $c_i=0$ for any $i$ and we have a contradiction.
Thus, assume $c_{N+1}\neq0$ and from $\langle w,\nu_j^{\bm\xi}\rangle=0$, it follows that
\begin{equation*}
	\frac{c_j}{c_{N+1}}=(-1)^{N-j+1}\frac{\|k_{N+1}^{\bm h}\|^2}{\|k_j^{\bm h}\|^2}, \qquad \qquad j=1,\dots,N.
\end{equation*}
Now, the assumption $\displaystyle\int_0^1 w\,dx=0$ implies
\begin{align*}
	0&=\sum_{j=1}^{N}\frac{c_j}{c_{N+1}}\int_0^1k_j^{\bm h}\,dx+\int_0^1k_{N+1}^{\bm h}\,dx\\
	&=\sum_{j=1}^{N}\frac{(-1)^{N-j+1}\|k_{N+1}^{\bm h}\|^2}{\|k_j^{\bm h}\|^2}\int_0^1k_j^{\bm h}\,dx+\int_0^1k_{N+1}^{\bm h}\,dx.
\end{align*}
Since 
\begin{equation*}
	\lim_{\e\to0}\frac{\|k_{N+1}^{\bm h}\|^2}{\|k_j^{\bm h}\|^2}=1 \qquad  \mbox{and} \qquad 
	 \lim_{\e\to0}\int_0^1k_j^{\bm h}\,dx=2(-1)^j, \quad \forall\, j=1,\dots,N+1,
\end{equation*}
we end up with $2N(-1)^{N+1}+2(-1)^{N+1}=0$, and we have a contradiction.

Therefore, $w$ satisfies condition \eqref{eq:condforcoerc} for some $\theta_0\in(0,\frac{\pi}{2}]$, and \eqref{eq:coercive} follows from \cite[Section 4.2]{Carr-Pego}.
\end{proof}
Thanks to the previous lemma, we can prove the following result.
\begin{prop}\label{prop:E>}
Let $F\in C^3(\R)$ be such that \eqref{eq:ass-F} holds and $g\in C^1(\R)$.
Given $N\in\mathbb{N}$ and $\delta\in(0,1/N+1)$, there exist $\varepsilon_0, C>0$, 
such that for $\varepsilon$ and $\rho$ satisfying \eqref{eq:triangle},
\begin{itemize}
	\item[(i)] if $(w,v,\bm\xi)\in\hat{\mathcal{Z}}_{{}_{\Gamma,\rho}}$, then
	\begin{equation}\label{eq:boundsforE}
		\begin{aligned}
			\tfrac18\Lambda_0\varepsilon\|w\|^2_{{}_{L^\infty}}+ \tfrac14\tau\|v\|^2\leq E^{\bm\xi}[w,v],\\
			\tfrac14\Lambda_0\|w\|^2 +\tfrac14\tau\|v\|^2\leq E^{\bm\xi}[w,v],\\
			E^{\bm\xi}[w,v]\leq C\Gamma\exp(-2A\ell^{\bm h}/\varepsilon),
		\end{aligned}
	\end{equation}
where $E^{\bm\xi}[w,v]$ is the functional defined in \eqref{eq:functional-Exi} and $\Lambda$ is the positive constant introduced in \eqref{eq:coercive};
	\item[(ii)] if $(w,v,\bm\xi)\in\hat{\mathcal{Z}}_{{}_{\Gamma,\rho}}$ is a solution to \eqref{eq:system-w-v-xi} for $t\in[0,T]$, then
	\begin{equation}\label{eq:boundforxi'}
		|\bm\xi'|\leq C\e^{1/2}\|v\|\leq C(\varepsilon/\tau)^{1/2}\exp(-A\ell^{\bm h}/\varepsilon).
	\end{equation}
	\end{itemize}
\end{prop}
\begin{proof}
We proceed as in the proof of \cite[Proposition 3.1]{FLM17}.
The proof of the three inequalities in \eqref{eq:boundsforE} is very similar; 
let us prove \eqref{eq:boundforxi'}.
Assume that $(w,v,\bm\xi)\in\hat{\mathcal{Z}}_{{}_{\Gamma,\rho}}$ is a solution to \eqref{eq:system-w-v-xi} for $t\in[0,T]$. 
To obtain an upper bound on $|\bm\xi'|$, consider the matrix $\hat{S}(\bm\xi,w)$.
We infer
\begin{equation*}
	|\langle w,\nu^{\bm\xi}_{ji}\rangle| \leq\|w\|\|\nu^{\bm\xi}_{ji}\|\leq C\e^{-3/2}\|w\|,
\end{equation*}
where we used the formula \eqref{eq:nuij-est}.
By using \eqref{eq:boundsforE}, we deduce that $\hat{S}(\bm\xi,w)$
satisfies the formula \eqref{eq:S-matrix} and the inverse matrix $\hat{S}^{-1}(\bm\xi,w)$ satisfies \eqref{eq:S^-1}.
Therefore, by applying $\hat{S}^{-1}(\bm\xi,w)$ to the third equation of \eqref{eq:system-w-v-xi} and using \eqref{eq:uj-est}, we conclude
\begin{equation*}
	|\bm\xi'|\leq\|\hat{S}^{-1}(\bm\xi,w)\||\bm Y(\bm\xi,v)|\leq C\e\|v\|\|\nu^{\bm\xi}_j\|\leq C\e^{1/2}\|v\|,
\end{equation*}
that is, the first inequality in \eqref{eq:boundforxi'}.
The second one follows from \eqref{eq:boundsforE}.
\end{proof}

Proposition \ref{prop:E>} states that if $(w,v,\bm\xi)\in\hat{\mathcal{Z}}_{{}_{\Gamma,\rho}}$ is a solution to \eqref{eq:system-w-v-xi} for $t\in[0,T]$,
then $\|w\|_{{}_{L^\infty}}$, $\|v\|$, and $|\bm\xi'|$ are exponentially small as $\e\to0$.
As a consequence, the solution $u$ to equation \eqref{eq:hyp-nonlocal} is well approximated by $u^{\bm\xi}\in\mathcal{M}$,
the $L^2$--norm of $u_t$ and the speed of the transition points are exponentially small.
Indeed, since $\bm h=(\bm\xi,z(\bm\xi))$ and $z$ satisfies \eqref{eq:derh_N+1}, by using \eqref{eq:boundforxi'} we obtain
\begin{equation}\label{eq:boundforhn+1'}
	|h'_{N+1}|=\left|\sum_{j=1}^Nz_j\xi'_j\right|\leq C\e^{1/2}\|v\|\leq C(\varepsilon/\tau)^{1/2}\exp(-A\ell^{\bm h}/\varepsilon),
\end{equation}
and all the $N+1$ layers move with an exponentially small speed.
\begin{rem}\label{rem:energy-tau}
Here, we do some comments on the choice of the functional $E^{\bm\xi}[w,v]$ \eqref{eq:functional-Exi}.
First, we mention that by (formally) taking $\tau=0$ in \eqref{eq:functional-Exi}, one obtains the functional used in \cite{Carr-Pego} 
to study the Allen--Cahn equation \eqref{eq:AC}, and that can be used to study the mass conserving Allen--Cahn equation \eqref{eq:maco-AC}.
In the latter case, the system describing the metastable dynamics of the solutions can be (formally) obtained 
by taking $\tau=0$ and $g\equiv1$ in the ODE-PDE coupled system \eqref{eq:system-w-v-xi}.
In particular, notice that if $\tau=0$ and $g\equiv1$, the second equation of \eqref{eq:system-w-v-xi} 
gives the expression for $v$, which has to be substituted in the equations for $w$ and $\bm\xi$.
Hence, the exponentially small velocity of the layers can be deduced by estimating all the terms appearing
in the equation for $\bm\xi$ (cfr. the estimates of Section \ref{sec:layerdyn}).
In the \emph{hyperbolic} case $\tau>0$, we add two terms in the definition \eqref{eq:functional-Exi}.
Similarly to the definition of the energy \eqref{eq:energy}, we add the term $\frac\tau2\|v\|^2$, which corresponds to the $L^2$--norm of the time derivative $u_t$.
The presence of the linear term $\e\tau\langle w,v\rangle$ is perhaps not so natural, 
but, as we will see in the following, it is crucial to prove that if $(w,v,\bm\xi)$ is a solution to the system \eqref{eq:system-w-v-xi},
belonging to $\hat{\mathcal{Z}}_{{}_{\Gamma,\rho}}$ for $t\in[0,T]$, then we have $E^{\bm\xi}[w,v]<\Gamma\Psi(\bm h)$ for any $t\in[0,T]$.
Thanks to this result, we can state that solutions can leave $\hat{\mathcal{Z}}_{{}_{\Gamma,\rho}}$ only if $\bm h\in\partial\Omega_\rho$,
and we have persistence in the slow channel for (at least) an exponentially long time because the layers move with an exponentially small velocity.
As we already mentioned in Remark \ref{rem:main-tau}, notice that the exponentially small velocity of the layers in the slow channel is a consequence
of the exponentially smallness of the $L^2$--norm of $v=u_t$.
\end{rem}
To obtain the lower bound of the time taken for the solution to leave $\hat{\mathcal{Z}}_{{}_{\Gamma,\rho}}$, we will use the following result.

\begin{prop}\label{prop:d/dtE}
Let $F\in C^3(\R)$ and $g\in C^1(\R)$ be such that \eqref{eq:ass-F} and \eqref{eq:ass-g} hold.
Given $N\in\mathbb{N}$ and $\delta\in(0,1/N+1)$, there exist $\Gamma_2>\Gamma_1>0$ and $\varepsilon_0>0$ 
such that if $\Gamma\in[\Gamma_1,\Gamma_2]$, $\varepsilon,\rho$ satisfy \eqref{eq:triangle} 
and $(w,v,\bm\xi)\in\hat{\mathcal{Z}}_{{}_{\Gamma,\rho}}$ is a solution to \eqref{eq:system-w-v-xi} for $t\in[0,T]$, then
for some $\eta\in(0,1)$, we have
\begin{equation}\label{E-GPsi^2}
	\frac{d}{dt}\bigl\{E^{\bm\xi}[w,v]-\Gamma\Psi(\bm h)\bigr\}
		\leq-\eta\,\varepsilon\bigl\{E^{\bm\xi}[w,v]-\Gamma\Psi(\bm h)\bigr\}
	\qquad\textrm{for}\quad t\in[0,T].
\end{equation}
\end{prop}

\begin{proof}
In all the proof, symbols $C, c, \eta$ denote generic positive constants, independent on $\varepsilon$,
and with $\eta\in(0,1)$.
Let $(w,v,\bm\xi)\in\hat{\mathcal{Z}}_{{}_{\Gamma,\rho}}$ be a solution to \eqref{eq:system-w-v-xi} for $t\in[0,T]$;
in particular, in the proof we shall use that $w$ and $v$ are functions of zero mean and satisfy inequalities \eqref{eq:boundsforE} and \eqref{eq:boundforxi'}.
Let us start by differentiating with respect to $t$ and estimating all the terms appearing in the functional $E^{\bm\xi}[w,v]$ \eqref{eq:functional-Exi}.
Regarding the first term, we have
\begin{equation*}
	\frac{d}{dt}\Bigl\{\tfrac12\langle w,L^{\bm\xi}w\rangle\Bigr\}=\langle w_t,L^{\bm\xi}w\rangle+\tfrac12\sum_{j=1}^N\xi'_j\langle w,f''(u^{\bm\xi})u_j^{\bm\xi}w\rangle,
\end{equation*}
and so, taking the inner product between the first equation of \eqref{eq:system-w-v-xi} and $L^{\bm\xi}w$, we infer
\begin{align*}
	\frac{d}{dt}\Bigl\{\tfrac12\langle w,L^{\bm\xi}w\rangle\Bigr\}&=\langle v,L^{\bm\xi}w\rangle-\sum_{j=1}^{N}\xi_j'\langle u^{\bm\xi}_j,L^{\bm\xi}w\rangle
		+\tfrac12\sum_{j=1}^N\xi'_j\langle w,f''(u^{\bm\xi})u_j^{\bm\xi}w\rangle\\
	&= \langle v,L^{\bm\xi}w\rangle-\sum_{j=1}^{N}\xi_j'\langle L^{\bm\xi}u^{\bm\xi}_j,w\rangle
		+\tfrac12\sum_{j=1}^N\xi'_j\langle w,f''(u^{\bm\xi})u_j^{\bm\xi}w\rangle\\
	&\leq\langle v,L^{\bm\xi}w\rangle+C\e^{1/2}\|v\|\|w\|\left(\max_j\|L^{\bm\xi}u^{\bm\xi}_j\|+\max_j\|u_j^{\bm\xi}\|_{{}_{L^\infty}}\|w\|\right),
\end{align*}
where in the last passage we used the first inequality of \eqref{eq:boundforxi'} and H\"older inequality.
Since
\begin{equation*}
	\|L^{\bm\xi}u_j^{\bm\xi}\|=\|L^{\bm\xi}u_j^{\bm h}+z_jL^{\bm\xi}u_{N+1}^{\bm h}\|\leq C\e^{-1/2}\exp(-Al^{\bm h}/\e), \qquad \qquad j=1,\dots,N,
\end{equation*}
where we used \cite[Proposition 7.2]{Carr-Pego2}, and
\begin{equation*}
	\|u_j^{\bm\xi}\|_{{}_{L^\infty}}\|w\|\leq C\e^{-3/2}\sqrt{\Gamma}\exp(-A\ell^{\bm h}/\varepsilon), \qquad \qquad j=1,\dots,N,
\end{equation*}
because of \eqref{eq:uj-est}-\eqref{eq:boundsforE}, by using Young's inequality and \eqref{eq:triangle},  we conclude
\begin{equation}\label{eq:1termE}
	\frac{d}{dt}\Bigl\{\tfrac12\langle w,L^{\bm\xi}w\rangle\Bigr\}\leq\langle v,L^{\bm\xi}w\rangle+C_\Gamma\exp(-c/\e)\|w\|^2+\eta\|v\|^2,
\end{equation}
where $C_\Gamma$ depends on $\Gamma$, but it is independent on $\e$.
For what concerns the second term in the energy $E^{\bm\xi}[w,v]$ \eqref{eq:functional-Exi}, 
taking the inner product between the second equation of \eqref{eq:system-w-v-xi} and $v$, we deduce
\begin{align*}
	\frac{d}{dt}\Bigl\{\tfrac12\tau\|v\|^2\Bigr\}& =\langle \tau v_t,v\rangle
		=\langle\mathcal{L}(u^{\bm\xi})-L^{\bm\xi}w -f_2w^2-g(u^{\bm\xi}+w)v,v\rangle\\
		&\qquad -\int_0^1 f(u^{\bm\xi}+w)\,dx\int_0^1v\,dx-\int_0^1[1-g(u^{\bm\xi}+w)]v\,dx\int_0^1v\,dx.
\end{align*}
Since $v$ is a function of zero mean, Young's inequality and the assumption on $g$ \eqref{eq:ass-g} yield 
\begin{equation}\label{eq:2termE}
	\begin{aligned}
		\frac{d}{dt}\Bigl\{\tfrac12\tau\|v\|^2\Bigr\}&\leq-\langle L^{\bm\xi}w,v\rangle+\|\mathcal{L}(u^{\bm\xi})\|\|v\|+C\|w\|_{{}_{L^\infty}}\|w\|\|v\|-\sigma\|v\|^2\\
		&\leq -\langle L^{\bm\xi}w,v\rangle-(\sigma-\eta)\|v\|^2+C\|w\|_{{}_{L^\infty}}^2\|w\|^2+C\|\mathcal{L}(u^{\bm\xi})\|^2\\
		&\leq -\langle L^{\bm\xi}w,v\rangle-(\sigma-\eta)\|v\|^2+C_\Gamma\exp(-c/\e)\|w\|^2+C\|\mathcal{L}(u^{\bm\xi})\|^2,	
	\end{aligned}
\end{equation}
where we again used \eqref{eq:boundsforE} and \eqref{eq:triangle}.

Finally, we estimate the time derivative of the scalar product $\langle w,\tau v\rangle$ as follows
\begin{align*}
	\frac{d}{dt}\langle w,\tau v\rangle & =\langle w_t,\tau v\rangle+\langle w,\tau v_t\rangle\\
	&=\langle v-\sum_{j=1}^{N}u^{\bm\xi}_j\xi_j',\tau v\rangle+\langle w,\mathcal{L}(u^{\bm\xi})-L^{\bm\xi}w -f_2w^2-g(u^{\bm\xi}+w)v\rangle\\
	&\qquad\qquad -\int_0^1 w\,dx\int_0^1 f(u^{\bm\xi}+w)\,dx-\int_0^1 w\,dx\int_0^1[1-g(u^{\bm\xi}+w)]v\,dx.
\end{align*}
Since $w$ is a function of zero mean, one has
\begin{equation}\label{eq:3termE}
	\begin{aligned}
		\frac{d}{dt}\langle w,\tau v\rangle &\leq\tau\|v\|^2+C\tau\e^{1/2}\max_j\|u^{\bm\xi}_j\|\|v\|^2-\langle w,L^{\bm\xi}w\rangle+C\e\|w\|^2\\
	&\qquad \qquad +\varepsilon^{-1}\|\mathcal{L}(u^{\bm\xi})\|^2+C\|w\|_{{}_{L^\infty}}\|w\|^2+\e^{-1}\eta\|v\|^2\\
	&\leq -\langle w,L^{\bm\xi}w\rangle+C(\varepsilon+\|w\|_{{}_{L^\infty}})\|w\|^2
		+(C+\eta\,\varepsilon^{-1})\|v\|^2+\varepsilon^{-1}\|\mathcal{L}(u^{\bm\xi})\|^2,
	\end{aligned}
\end{equation}
where, in particular, the inequalities
\begin{equation*}
	\begin{aligned}
	\langle w,\mathcal{L}(u^{\bm\xi})\rangle&\leq \tfrac12\varepsilon\|w\|^2+\tfrac12\varepsilon^{-1}\|\mathcal{L}(u^{\bm\xi})\|^2,\\
	\langle w,g(u^{\bm\xi}+w,\tau)v\rangle&\leq C\varepsilon\|w\|^2+\eta\,\varepsilon^{-1}\|v\|^2
	\end{aligned}
\end{equation*}
have been used.
Collecting the estimates \eqref{eq:1termE}, \eqref{eq:2termE} and \eqref{eq:3termE}, we get
\begin{equation*}
	\begin{aligned}
	\frac{dE^{\bm\xi}}{dt}[w,v] &\leq -\varepsilon\langle w,L^{\bm\xi}w\rangle-[\sigma-C\varepsilon-3\eta]\|v\|^2\\
	&\hskip1.0cm  +C_\Gamma\bigl\{\exp(-c/\varepsilon)
		+\varepsilon (\varepsilon+\|w\|_{{}_{L^\infty}})\bigr\}\|w\|^2+(C+1)\|\mathcal{L}(u^{\bm\xi})\|^2\\
	&\leq -\varepsilon\langle w,L^{\bm\xi}w\rangle-\eta\sigma\|v\|^2+C_\Gamma\varepsilon\bigl\{\exp(-c/\varepsilon)+\varepsilon\bigr\}\|w\|^2+C\|\mathcal{L}(u^{\bm\xi})\|^2,
	\end{aligned}
\end{equation*}
for $\varepsilon$ and $\eta$ small.
Thus, by using \eqref{eq:coercive} and the following estimate
\begin{equation}\label{L(u^h)<Psi}
	\|\mathcal{L}(u^{\bm h})\|^2\leq C\varepsilon\,\Psi(\bm h)\leq C\varepsilon\exp(-2A\ell^{\bm h}/\varepsilon),
\end{equation}
see \cite[Theorem 3.5]{Carr-Pego}, we obtain
\begin{equation*}
	\frac{dE^{\bm\xi}}{dt}[w,v]\leq -\varepsilon\bigl\{1-C_\Gamma\bigl(\exp(-c/\varepsilon)+\varepsilon\bigr)\bigr\}\langle w,L^{\bm\xi}w\rangle
		-\eta\sigma\|v\|^2+C\varepsilon\Psi,
\end{equation*} 
Hence, for $\varepsilon\in(0,\varepsilon_0)$, with $\varepsilon_0$ small (and dependent on $\Gamma$), we deduce the bound
\begin{equation*}
	1- C_\Gamma\bigl(\exp(-c/\varepsilon)+\varepsilon\bigr)\geq\eta.
\end{equation*}
Substituting, we infer
\begin{equation*}
	\begin{aligned}
	\frac{dE^{\bm\xi}}{dt}[w,v] &\leq -\eta\,\varepsilon\langle w,L^{\bm\xi}w\rangle-\eta\sigma\|v\|^2+C\varepsilon\Psi\\
	&\leq -\eta\,\varepsilon E^{\bm\xi}[w,v]
		-\tfrac12\eta\,\varepsilon\langle w,L^{\bm\xi}w\rangle+\eta\,\varepsilon^2\tau\langle w,v\rangle
		-\eta\bigl(\sigma-\tfrac12\varepsilon\tau\bigr)\|v\|^2+C\varepsilon\Psi\\
	&\leq -\eta\,\varepsilon E^{\bm\xi}[w,v]
		-\tfrac12\eta\,\varepsilon\bigl(1-C\varepsilon\tau\bigr)\langle w,L^{\bm\xi}w\rangle
		-\eta\bigl(\sigma-C\varepsilon\tau\bigr)\|v\|^2+C\varepsilon\Psi,
	\end{aligned}
\end{equation*} 
again from Young's inequality and \eqref{eq:coercive}.
Finally, for $\varepsilon_0$ sufficiently small, we end up with
\begin{equation}\label{eq:E'}
	\frac{dE^{\bm\xi}}{dt}[w,v]\leq -\eta\,\varepsilon E^{\bm\xi}[w,v]-\eta\sigma\|v\|^2+C\varepsilon\Psi.
\end{equation} 
Now, let us consider the term $\Psi(\bm h)$; direct differentiation gives
\begin{equation*}
	\frac{d\Psi}{dt}=2\sum_{i,j=1}^{N+1}\langle\mathcal{L}(u^{\bm\xi}),k^{\bm h}_j\rangle\Bigl\{\langle\mathcal{L}(u^{\bm\xi}),k_{ji}^{\bm h}\rangle
		-\langle L^{\bm\xi}u^{\bm h}_i,k^{\bm h}_j\rangle\Bigr\}h'_i.
\end{equation*}
Using the estimates provided by \cite[Proposition 7.2]{Carr-Pego2} and by \eqref{eq:boundforxi'}, \eqref{eq:boundforhn+1'}, we have
\begin{equation*}
	\begin{aligned}
	\left|h'_i\langle\mathcal{L}(u^{\bm\xi}),k_{ji}^{\bm h}\rangle\right|
	&\leq|\bm h'|_{{}_\infty}\|\mathcal{L}(u^{\bm\xi})\|\|k^{\bm h}_{ji}\|
		\leq C\varepsilon^{-1}\|\mathcal{L}(u^{\bm\xi})\|\|v\|,\\
	\left|h'_i\langle L^{\bm\xi}u^{\bm h}_i,k^{\bm h}_j\rangle\right|
	&\leq|\bm h'|_{{}_\infty}\|k^{\bm h}_j\|\|L^{\bm h}u^{\bm h}_i\|
		\leq C\exp(-c/\varepsilon)\|v\|,
	\end{aligned}
\end{equation*}
for any $i,j=1,\dots,N+1$. 
Therefore, observing that $|\langle\mathcal{L}(u^{\bm\xi}),k^{\bm h}_j\rangle|\leq C\varepsilon^{-1/2} \|\mathcal{L}(u^{\bm\xi})\|$,
we infer the bound
\begin{equation*}
	\left|\frac{d\Psi}{dt}\right|
	\leq C\varepsilon^{-1/2} \left\{\varepsilon^{-1}\|\mathcal{L}(u^{\bm\xi})\|+\exp(-c/\varepsilon)\right\}
		\|\mathcal{L}(u^{\bm\xi})\|\|v\|.
\end{equation*}
Using the inequality \eqref{L(u^h)<Psi}, we obtain
\begin{equation*}
	\begin{aligned}
	\left|\Gamma\frac{d\Psi}{dt}\right|
	&\leq C\,\Gamma\,\varepsilon^{-1/2}\bigl\{\Psi^{1/2}+\exp(-c/\varepsilon)\bigr\}\|v\|\Psi^{1/2}\\
	&\leq \eta\|v\|^2+C\,\Gamma^2\varepsilon^{-1}\bigl\{\Psi^{1/2}+\exp(-c/\varepsilon)\bigr\}^2\Psi.
	\end{aligned}
\end{equation*}
Hence, observing that $\Psi\leq C\exp\bigl(-c/\varepsilon\bigr)$, we end up with
\begin{equation}\label{eq:Psi'}
	\left|\Gamma\frac{d\Psi}{dt}\right|\leq \eta\|v\|^2+C\,\Gamma^2\exp(-c/\varepsilon)\Psi.
\end{equation}
In conclusion, combining the estimates \eqref{eq:E'} and \eqref{eq:Psi'}, we obtain that if 
$(w,v,\bm\xi)\in\hat{\mathcal{Z}}_{{}_{\Gamma,\rho}}$ is a solution to \eqref{eq:system-w-v-xi}, then
\begin{equation*}
	\frac d{dt}\bigl\{E^{\bm\xi}[w,v]-\Gamma\Psi(\bm h)\bigr\} \leq
	-\eta\,\varepsilon E^{\bm\xi}[w,v]+C\bigl(\varepsilon+\Gamma^2\exp(-c/\varepsilon)\bigr)\Psi,
\end{equation*}
for some $\eta\in(0,1)$.
Therefore the estimate \eqref{E-GPsi^2} follows from 
\begin{equation*}
	C\exp(-c/\varepsilon)\Gamma^2-\eta\,\varepsilon\Gamma +C\varepsilon\leq 0,
\end{equation*}
 and the latter is verified for $\Gamma\in [\Gamma_1,\Gamma_2]$, provided $\varepsilon\in(0,\varepsilon_0)$, 
 with $\varepsilon_0$ sufficiently small so that $\eta^2\varepsilon-4C^2\exp(-c/\varepsilon)>0$.
\end{proof}

We stress that in the estimates \eqref{eq:1termE}-\eqref{eq:2termE}-\eqref{eq:3termE} it is fundamental that $w$ and $v$ are functions of zero mean,
and so, Theorem \ref{thm:existence-coord} is crucial in the proof of the persistence of the solution to \eqref{eq:hyp-nonlocal}-\eqref{eq:Neumann} in the slow channel. 
Now, we have all the tools needed to prove Theorem \ref{thm:main}.

\begin{proof}[Proof of Theorem \ref{thm:main}]
First of all, we define the slow channel $\mathcal{Z}_{{}_\rho}$.
Fix $\Gamma\in[\Gamma_1,\Gamma_2]$, $\e_0>0$ small and $\e,\rho$ satisfying \eqref{eq:triangle}
so that Proposition \ref{prop:d/dtE} holds. 
Then, the slow channel is
\begin{equation}\label{eq:slowchannel}
	\begin{aligned}
		\mathcal{Z}_{{}_\rho}:=\bigl\{(u,v)\,:\;u=u^{\bm\xi}+w,\;\;  (w,v)\in W\times V, \quad \bm\xi & \, \mbox{ is such that }\, 
		\bm h=(\bm\xi,z(\bm\xi))\in\bar\Omega_\rho, \\ 
		&\qquad\mbox{ and } \; E^{\bm{\xi}}[w,v]\leq\Gamma \Psi({\bm h})\bigr\}.
	\end{aligned}
\end{equation}
Assume that the initial data $(u_0,u_1)\in\,\stackrel{\circ}{\mathcal{Z}}_{{}_\rho}$, which means $u_0=u^{\bm h_0}+w_0$, $u_1=v_0$, 
with $\bm h_0\in\Omega_\rho$ and  $E^{\bm{\xi}}[w_0,v_0]<\Gamma \Psi({\bm h})$.
Notice that the estimates \eqref{eq:boundsforE} and the smallness of $\e$ ensure that the assumptions \eqref{eq:w-excoord} of Theorem \ref{thm:existence-coord}
are satisfied and we have the decomposition $u=u^{\bm\xi}+w$.
Studying the dynamics inside the slow channel \eqref{eq:slowchannel}
is equivalent to study the dynamics of the ODE-PDE coupled system \eqref{eq:system-w-v-xi} in the set $\hat{\mathcal{Z}}_{{}_{\Gamma,\rho}}$.
The estimates \eqref{eq:umenouh}-\eqref{eq:|h'|<exp-intro} inside the slow channel $\mathcal{Z}_{{}_\rho}$ follow from \eqref{eq:boundsforE} and \eqref{eq:boundforxi'}.
Let us give a lower bound on the time taken for the solution to leave the slow channel.
Assume that $(u,v)\in\mathcal{Z}_{{}_\rho}$ for $t\in[0,T_\varepsilon]$, where $T_\varepsilon$ is maximal.
The boundary of $\mathcal{Z}_{{}_\rho}$ consists of two parts: the ``ends'' where ${\bm h}\in\partial\Omega_\rho$,
meaning $h_j-h_{j-1}=\varepsilon/\rho$ for some $j$ and ``sides'' where $E^{\bm\xi}[w,v]=\Gamma\Psi({\bm h})$.
Thanks to Proposition \ref{prop:d/dtE}, we can state that the solution can leave $\mathcal{Z}_{{}_\rho}$ only through the ends.
Indeed, from \eqref{E-GPsi^2} it follows that
\begin{equation*}
	\frac d{dt}\Bigl\{\exp(\eta\,\varepsilon t)(E^{\bm\xi}[w,v]-\Gamma\Psi(\bm h))\Bigr\}\leq0,
	\quad \qquad t\in[0,T_\varepsilon]
\end{equation*}
and so,
\begin{equation*}
	\exp(\eta\,\varepsilon t)\{E^{\bm\xi}[w,v]-\Gamma\Psi(\bm h)\}(t)\leq\{E^{\bm\xi}[w,v]-\Gamma\Psi(\bm h)\}(0)<0,
	\qquad \quad t\in[0,T_\varepsilon].
\end{equation*}
Therefore, the solution $(u,v)$ remains in the channel $\mathcal{Z}_{{}_\rho}$ while $\bm h\in\Omega_\rho$ 
and if $T_\varepsilon<+\infty$ is maximal, then $\bm h(T_\varepsilon)\in\partial\Omega_\rho$, that is
\begin{equation}\label{hfrontiera}
	h_j(T_\varepsilon)-h_{j-1}(T_\varepsilon)=\varepsilon/\rho, \quad \qquad \textrm{for some } j\in\{1,\dots,N+2\}.
\end{equation}
Since the transition points move with an exponentially small velocity \eqref{eq:boundforxi'}-\eqref{eq:boundforhn+1'},
the solution $(u,v)$ remains in the channel for an exponentially long time. 
Precisely, from \eqref{eq:|h'|<exp-intro} we deduce
\begin{equation}\label{dhmax}
	|h_j(t)-h_j(0)|\leq C\left(\varepsilon/\tau\right)^{1/2}\exp(-A\ell^{\bm h(t)}/\varepsilon)t \qquad \textrm{for any } j=1,\dots,N+1,
\end{equation} 
for all $t\in[0,T_\varepsilon]$, where $\ell^{\bm h(t)}$ is the minimum distance between layers at the time $t$.
Combining \eqref{hfrontiera} and \eqref{dhmax},  we obtain 
\begin{equation*}
	\varepsilon/\rho\geq \ell^{\bm h(0)}-2C(\varepsilon/\tau)^{1/2}\exp(-A/\rho)T_\varepsilon.
\end{equation*}
Hence, by using \eqref{eq:triangle} we obtain
\begin{equation*}
	T_\varepsilon\geq C\bigl(\ell^{\bm h(0)}-\varepsilon/\rho\bigr)(\varepsilon/\tau)^{-1/2}\exp(A/\rho)\geq 
	C\bigl(\ell^{\bm h(0)}-\varepsilon/\rho\bigr)(\varepsilon/\tau)^{-1/2}\exp(A\delta/\varepsilon),
\end{equation*}
and the proof is complete.
\end{proof}

\section{Layer dynamics}\label{sec:layerdyn}
In this section, we derive the ODEs describing the exponentially slow motion of the $N+1$ layers.
We reason as in the derivation of the ODEs for the layer dynamics in \cite{FLM17,FLMpre}.
Since  $w$ is very small, we use the approximation $w\approx0$ in \eqref{eq:xi-ort} and then
\begin{equation}\label{eq:xi'-w=0}
	\sum_{i=1}^N\langle u^{\bm\xi}_i,\nu^{\bm\xi}_j\rangle\xi'_i=\langle v,\nu^{\bm\xi}_j\rangle, \qquad j=1,\dots,N.
\end{equation}
In order to eliminate $v$, let us differentiate and multiply by $\tau$ equation \eqref{eq:xi'-w=0}. 
We have
\begin{align*}
	\tau\sum_{i,l=1}^N \bigl(\langle u^{\bm\xi}_{il},\nu^{\bm\xi}_j\rangle+\langle u^{\bm\xi}_i,\nu^{\bm\xi}_{jl}\rangle\bigr)\xi'_l\xi'_i
	+\tau\sum_{i=1}^N\langle u^{\bm\xi}_i,\nu^{\bm\xi}_j\rangle\xi''_i=&\langle\mathcal{L}(u^{\bm\xi}),\nu^{\bm\xi}_j\rangle-\langle g(u^{\bm\xi})v,\nu^{\bm\xi}_j\rangle\\
	& -\int_0^1\nu^{\bm\xi}_j\,dx\int_0^1 f(u^{\bm\xi})\,dx\\
	&-\int_0^1\nu^{\bm\xi}_j\,dx\int_0^1[1-g(u^{\bm\xi})]v\,dx\\
	&+\tau\sum_{l=1}^N\langle v,\nu^{\bm\xi}_{jl}\rangle\xi'_l,
\end{align*}
for $j=1,\dots,N$.
Using the approximation $v\approx\displaystyle\sum_{i=1}^Nu^{\bm\xi}_i\xi'_i$, we obtain
\begin{align*}
	\tau\sum_{i,l=1}^N \bigl(\langle u^{\bm\xi}_{il},\nu^{\bm\xi}_j\rangle+\langle u^{\bm\xi}_i,\nu^{\bm\xi}_{jl}\rangle\bigr)\xi'_l\xi'_i
	+\tau\sum_{i=1}^N\langle u^{\bm\xi}_i,\nu^{\bm\xi}_j\rangle\xi''_i=
	&\langle\mathcal{L}(u^{\bm\xi}),\nu^{\bm\xi}_j\rangle-\sum_{i=1}^N\langle g(u^{\bm\xi})u^{\bm\xi}_i,\nu^{\bm\xi}_j\rangle\xi'_i\\\
	& -\int_0^1\nu^{\bm\xi}_j\,dx\int_0^1 f(u^{\bm\xi})\,dx\\
	&-\int_0^1\nu^{\bm\xi}_j\,dx\sum_{i=1}^N\xi'_i\int_0^1[1-g(u^{\bm\xi})]u^{\bm\xi}_i\,dx\\
	&+\tau\sum_{i,l=1}^N\langle u^{\bm\xi}_i,\nu^{\bm\xi}_{jl}\rangle\xi'_l\xi'_i,
\end{align*}
for $j=1,\dots,N$.
Let us denote by $\nabla^2_{\bm\xi}u^{\bm\xi}$ the Hessian of $u^{\bm\xi}$ with respect to $\bm\xi$ and 
by $q(\bm\upsilon):=\displaystyle\sum_{i,l=1}^N u^{\bm\xi}_{il}\upsilon_l \upsilon_i$ the quadratic form associated to $\nabla^2_{\bm\xi}u^{\bm\xi}$. 
Simplifying, we get
\begin{equation}\label{h-eq}
	\begin{aligned}
		\tau\sum_{i=1}^N\langle u^{\bm\xi}_i,\nu^{\bm\xi}_j\rangle\xi''_i+\sum_{i=1}^N\langle g(u^{\bm\xi})u^{\bm\xi}_i,\nu^{\bm\xi}_j\rangle\xi'_i
		+\tau\langle q(\bm\xi'),\nu^{\bm\xi}_j\rangle =&\langle\mathcal{L}(u^{\bm\xi}),\nu^{\bm\xi}_j\rangle-\int_0^1\nu^{\bm\xi}_j\,dx\int_0^1 f(u^{\bm\xi})\,dx\\
		&-\int_0^1\nu^{\bm\xi}_j\,dx\sum_{i=1}^N\xi'_i\int_0^1[1-g(u^{\bm\xi})]u^{\bm\xi}_i\,dx,
	\end{aligned}
\end{equation}
for $ j=1,\dots,N$. 
Let us rewrite equations \eqref{h-eq} in the compact form
\begin{equation}\label{eq:xi-vect}
	\tau S(\bm\xi)\bm\xi''+\mathcal{G}(\bm\xi)\bm\xi' +\tau\bm{\mathcal{Q}}(\bm\xi,\bm\xi')=\bm{\mathcal{P}}(\bm\xi)-\mathcal{R}(\bm\xi)\bm\xi',
\end{equation}	
where the matrix $S$ has the form \eqref{eq:S-matrix}, the matrices $\mathcal{G},\mathcal{R}\in\mathbb{R}^{N\times N}$ are defined by
\begin{equation*}
	\mathcal{G}_{ji}(\bm\xi):=\langle g(u^{\bm\xi})u^{\bm\xi}_i,\nu^{\bm\xi}_j\rangle, \qquad \qquad 
	\mathcal{R}_{ji}(\bm\xi):=\int_0^1\nu^{\bm\xi}_j\,dx\int_0^1[1-g(u^{\bm\xi})]u^{\bm\xi}_i\,dx,
\end{equation*}
and the vectors $\bm{\mathcal{Q}},\bm{\mathcal{P}}\in\mathbb{R}^{N}$ are given by
\begin{equation*}
	\mathcal{Q}_j(\bm\xi,\bm\xi'):=\langle q(\bm\xi'),\nu^{\bm\xi}_j\rangle, \qquad \qquad 
	\mathcal{P}_j(\bm\xi):=\langle\mathcal{L}(u^{\bm\xi}),\nu^{\bm\xi}_j\rangle-\int_0^1\nu^{\bm\xi}_j\,dx\int_0^1 f(u^{\bm\xi})\,dx.
\end{equation*}
We want to identify the leading terms in \eqref{eq:xi-vect}, having in mind the estimates for $u^{\bm\xi}$, $\nu_j^{\bm\xi}$  and their derivatives;
namely we shall rewrite $\mathcal{G}$, $\bm{\mathcal{Q}}$, $\bm{\mathcal{P}}$ and $\mathcal{R}$ 
by neglecting the exponentially small remainders in the asymptotic expansion for $\varepsilon\to 0$.

Let us start with the matrix $\mathcal{G}$ and use \cite[Proposition 4.1]{FLM17}, 
which states that if $\rho$ is sufficiently small and $\bm h\in\Omega_\rho$, then there exists $C>0$ such that, 
\begin{equation}\label{g(u^h)u_j,k_j}
	\begin{aligned}
		&\bigl|\langle g(u^{\bm h})u_j^{\bm h},k^{\bm h}_j\rangle-\varepsilon^{-1}C_{F,g}\bigr|
		\leq C\varepsilon^{-1}\max\{\beta^{j-1/2},\beta^{j+1/2}\}
		\leq C\varepsilon^{-1}\exp(-A\ell^{\bm h}/2\varepsilon), \\
		&\bigl|\langle g(u^{\bm h})u_j^{\bm h},k^{\bm h}_{j+1}\rangle\bigr|+\bigl|\langle g(u^{\bm h})u_{j+1}^{\bm h},k^{\bm h}_j\rangle\bigr|
		\leq C\varepsilon^{-1}\beta^{j+1/2}
		\leq C\varepsilon^{-1}\exp(-A\ell^{\bm h}/2\varepsilon),\\ 
		&\langle g(u^{\bm h})u_j^{\bm h},k^{\bm h}_i\rangle=0 \qquad\mbox{ if } |j-i|>1.
	\end{aligned}
\end{equation}
where 
\begin{equation*}	
	C_{F,g}:=\int_{-1}^1\sqrt{2F(s)}g(s)ds.
\end{equation*}
From the definitions of $\mathcal{G}_{ji}(\bm\xi)$, $u^{\bm\xi}_i$ and $\nu^{\bm\xi}_j$, it follows that
\begin{align*}
	\mathcal{G}_{ji}(\bm\xi)=&\langle g(u^{\bm\xi})u^{\bm h}_i,k^{\bm h}_j\rangle+(-1)^{N-j}\langle g(u^{\bm\xi})u^{\bm h}_i,k^{\bm h}_{N+1}\rangle\\
	&\qquad +z_i\langle g(u^{\bm\xi})u^{\bm h}_{N+1},k^{\bm h}_j\rangle+(-1)^{N-j}z_i\langle g(u^{\bm\xi})u^{\bm h}_{N+1},k^{\bm h}_{N+1}\rangle,
\end{align*}
and, by using \eqref{eq:derh_N+1}, \eqref{eq:triangle} and the estimates \eqref{g(u^h)u_j,k_j},
we obtain the following formula for the matrix $\mathcal{G}$: 
\begin{equation*}
	\mathcal{G}(\bm\xi)=\frac{C_{F,g}}{\e}\left(\begin{array}{ccccc}  2 & -1 & 1 & \dots & (-1)^{N+1}\\
	-1 & 2 & -1 & \dots & (-1)^{N}\\
	1 & -1 & 2 & \dots & (-1)^{N+1}\\
	\dots & \dots & \dots & \dots & \dots\\
	(-1)^{N+1}& (-1)^{N} & (-1)^{N+1} & \dots & 2
	\end{array}\right)+\mathcal{O}\left(\exp(-c/\e)\right),
\end{equation*}
for some positive constant $c$ (independent on $\e$).
Therefore, we have
\begin{equation}\label{eq:G-matrix}
	\mathcal{G}(\bm\xi)=\gamma_{{}_{F,g}} S(\bm\xi)+\mathcal{O}(\exp(-c/\e)),
\end{equation}  
where $S(\bm\xi)$ satisfies \eqref{eq:S-matrix} and $\gamma_{{}_{F,g}}$ is the constant introduced in Section \ref{sec:st-main}:
\begin{equation*}
	\gamma_{{}_{F,g}}:=\frac{C_{F,g}}{c_{{}_F}}=\frac{\displaystyle\int_{-1}^1\sqrt{F(s)}g(s)\,ds}{\displaystyle\int_{-1}^1\sqrt{F(s)}\,ds}.
\end{equation*} 
Now, let us focus our attention on the term $\tau\bm{\mathcal{Q}}(\bm\xi,\bm\xi')$;
one has
\begin{align*}
	\mathcal{Q}_j(\bm\xi,\bm\xi')&=\sum_{i,l=1}^{N}\langle  u_{il}^{\bm\xi},\nu^{\bm\xi}_j\rangle\xi'_i\xi'_l\\
	&=\sum_{i,l=1}^{N}\langle u_{il}^{\bm h}+z_lu_{i,N+1}^{\bm h}+z_iu_{N+1,l}^{\bm h}+z_iz_lu_{N+1,N+1}^{\bm h},k^{\bm h}_j+(-1)^{N-j}k^{\bm h}_{N+1}\rangle\xi'_i\xi'_l.
\end{align*}
All the elements $\langle u^{\bm h}_{il},k^{\bm h}_j\rangle$ have been estimated in \cite[Section 4]{FLM17} and 
we have $\langle u^{\bm h}_{il},k^{\bm h}_j\rangle=\mathcal{O}\left(\exp(-c/\e)\right)$ for any $i,l,j$, and then
\begin{equation}\label{eq:Q-vector}
	\mathcal{Q}_j(\bm\xi,\bm\xi')=\mathcal{O}(\exp(-c/\e))\sum_{i,l=1}^{N}\xi'_i\xi'_l, \qquad \qquad j=1,\dots,N.
\end{equation}
It remains to identify the leading terms in the right hand side of \eqref{eq:xi-vect};
concerning the first term appearing in $\bm{\mathcal{P}}(\bm\xi)$, we have 
\begin{equation}\label{eq:leadP}
	\begin{aligned}
		\langle\mathcal{L}(u^{\bm\xi}),\nu^{\bm\xi}_j\rangle&=\langle\mathcal{L}(u^{\bm\xi}),k^{\bm h}_j\rangle+(-1)^{N-j}\langle\mathcal{L}(u^{\bm\xi}),k^{\bm h}_{N+1}\rangle\\
		&=\alpha^{j+1}-\alpha^{j}+(-1)^{N-j}\left(\alpha^{N+2}-\alpha^{N+1}\right),  
	\end{aligned}
\end{equation}
for $j=1,\dots,N$, where we used the definition \eqref{eq:newtangvec} and \cite[Lemma 3.3]{Carr-Pego}.
In the next result, we give an estimate on $\int_0^1 f(u^{\bm\xi})\,dx$.
\begin{lem}
Let $f=-F'$ with $F$ satisfying \eqref{eq:ass-F} and $u^{\bm\xi}\in\mathcal{M}$ defined by \eqref{eq:uh}.
Then,  
\begin{equation}\label{eq:intf}
	\left|\int_0^1 f(u^{\bm\xi})\,dx\right|\leq C\e\sum_{i=1}^{N+1}\left|\alpha^i-\alpha^{i+1}\right|.
\end{equation}
\end{lem}
\begin{proof}
From the definition \eqref{eq:uh}, it follows that
\begin{align*}
	\int_0^1 f(u^{\bm\xi})\,dx= \sum_{i=1}^{N+1}\int_{I_j}f(u^{\bm\xi})\,dx=&\sum_{i=1}^{N+1}\Bigg[\int_{h_{j-1/2}}^{h_j}f(\phi^j+\chi^j(\phi^{j+1}-\phi^j))\,dx+\\
	&\qquad\quad\int_{h_j}^{h_{j+1/2}}f\left(\phi^{j+1}+(1-\chi^j)(\phi^j-\phi^{j+1})\right)\,dx\Bigg].
\end{align*}
Since for $x\in[h_j-\e,h_j+\e]$ it holds
\begin{equation*}
	|\phi^j(x)-\phi^{j+1}(x)|\leq C|\alpha^j-\alpha^{j+1}|, \qquad\qquad j=1,\dots,N+1,
\end{equation*}
for some $C>0$ independent on $\e$ (see \cite[Lemma 8.2]{Carr-Pego}), we split
\begin{align*}
	\int_{I_j}f(u^{\bm\xi})\,dx&=\int_{h_{j-1/2}}^{h_j-\e}f(\phi^j)\,dx+\int_{h_j-\e}^{h_j}\left[f(\phi^j)+f'(\zeta_{j_1})\chi^j(\phi^{j+1}-\phi^j)\right]\,dx\\
	&\;+\int_{h_j}^{h_j+\e}\left[f(\phi^{j+1})+f'(\zeta_{j_2})(1-\chi^j)(\phi^j-\phi^{j+1})\right]\,dx+\int_{h_j+\e}^{h_{j+1/2}}f(\phi^{j+1})\,dx,
\end{align*}
where we used the definition of $\chi^j$, 
and we obtain
\begin{equation*}
	\left|\int_{I_j}f(u^{\bm\xi})\,dx\right|\leq\left|\int_{h_{j-1/2}}^{h_j}f(\phi^j)\,dx+\int_{h_j}^{h_{j+1/2}}f(\phi^{j+1})\,dx\right|+C\e|\alpha^j-\alpha^{j+1}|,
\end{equation*}
for $j=1,\dots,N+1$.
However, by definition $\e^2\phi^j_{xx}+f(\phi^j)=0$ \eqref{eq:fi}, and so
\begin{equation*}
	\left|\int_{I_j}f(u^{\bm\xi})\,dx\right|\leq\e^2\left|\phi^j_x(h_j)-\phi^{j+1}_x(h_j)\right|+C\e|\alpha^j-\alpha^{j+1}|,
\end{equation*}
for $ j=1,\dots,N+1$.
By using \cite[Lemma 8.2, estimate (8.2)]{Carr-Pego}, we end up with 
\begin{equation*}
	\left|\int_{I_j}f(u^{\bm\xi})\,dx\right|\leq C\e|\alpha^j-\alpha^{j+1}|, \qquad\qquad j=1,\dots,N+1,
\end{equation*}
and, as a trivial consequence we conclude \eqref{eq:intf}.
\end{proof}

Combining \eqref{eq:int-nu}, \eqref{eq:leadP} and \eqref{eq:intf}, we deduce that the leading term in $\bm{\mathcal{P}}(\bm\xi)$ is
\begin{equation*}
	\mathcal{P}^*_j(\bm\xi):=\alpha^{j+1}-\alpha^{j}+(-1)^{N-j}\left(\alpha^{N+2}-\alpha^{N+1}\right), \qquad \qquad j=1,\dots,N.
\end{equation*}
Indeed, for \eqref{eq:int-nu}, \eqref{eq:leadP} and \eqref{eq:intf} one has
\begin{equation}\label{eq:P-vector}
	\left|\bm{\mathcal{P}}(\bm\xi)-\bm{\mathcal{P}}^*(\bm\xi)\right|\leq C\exp\left(-c/\e\right)|\bm{\mathcal{P}}^*(\bm\xi)|.
\end{equation}
Finally, by using again \eqref{eq:int-nu} we infer
\begin{equation}\label{eq:R-matrix}
	\begin{aligned}
		|\mathcal{R}_{ji}(\bm\xi)|&=\left|\int_0^1\nu^{\bm\xi}_j\,dx\right|\left|\int_0^1[1-g(u^{\bm\xi})]u^{\bm\xi}_i\,dx\right|\\
		&\leq\left|\int_0^1\nu^{\bm\xi}_j\,dx\right|\|1-g(u^{\bm\xi})\|\|u^{\bm\xi}_i\|=\mathcal{O}\left(\exp(-c/\e)\right), \qquad \qquad i,j=1,\dots,N.
	\end{aligned}
\end{equation}
Taking into account \eqref{eq:G-matrix}, \eqref{eq:Q-vector}, \eqref{eq:P-vector}, \eqref{eq:R-matrix} and neglecting the exponentially smallest terms,
from \eqref{eq:xi-vect} we derive the following system of ODEs
\begin{equation*}
	\tau S(\bm\xi)\bm\xi''+\gamma_{{}_{F,g}} S(\bm\xi)\bm\xi'=\bm{\mathcal{P}}^*(\bm\xi).
\end{equation*}
By applying the inverse matrix $S^{-1}(\bm\xi)$, we end up with
\begin{equation*}
	\tau\bm\xi''+\gamma_{{}_{F,g}}\bm\xi'=S^{-1}(\bm\xi)\bm{\mathcal{P}}^*(\bm\xi).
\end{equation*}
Hence, using the formula \eqref{eq:S^-1} for $S^{-1}(\bm\xi)$, we obtain the following ODE for $\xi_j$
\begin{equation*}
	\tau \xi''_j+\gamma_{{}_{F,g}}\xi'_j=\frac{\e}{c_{{}_F}}\left(\alpha^{j+1}-\alpha^j+\frac{(-1)^{j+1}}{N+1}\sum_{i=1}^{N+1}(-1)^i\left(\alpha^{i+1}-\alpha^i\right)\right), 
\end{equation*}
for $j=1,\dots,N$.
Since $\bm\xi$ represents the vector of the first $N$ components of $\bm h$, we derived the ODEs for the first $N$ transition points;
to obtain the equation for $h_{N+1}$ we use the first equality in \eqref{eq:boundforhn+1'} and 
we neglect the exponentially smallest terms in \eqref{eq:derh_N+1}, namely we consider the approximation
\begin{equation*}
	h'_{N+1}\approx\sum_{j=1}^{N}(-1)^{N-j}\xi'_j, \qquad \qquad 
	h''_{N+1}\approx\sum_{j=1}^{N}(-1)^{N-j}\xi''_j.		
\end{equation*}
Thus, we get
\begin{align*}
	\tau h''_{N+1}+\gamma_{{}_{F,g}}h'_{N+1}&=\frac{\e}{c_{{}_F}}\left(\sum_{j=1}^{N}(-1)^{N-j}\left(\alpha^{j+1}-\alpha^j\right)+\sum_{j=1}^{N+1}\frac{N(-1)^{N+j+1}}{N+1}\left(\alpha^{j+1}-\alpha^j\right)\right)\\
	&=\frac{\e}{c_{{}_F}}\left(\alpha^{N+2}-\alpha^{N+1}+\frac{(-1)^N}{N+1}\sum_{j=1}^{N+1}(-1)^j\left(\alpha^{j+1}-\alpha^j\right)\right).
\end{align*}
We conclude that the dynamics of the transition points $(h_1,\dots,h_{N+1})$ is described by the ODEs \eqref{eq:ODE-hypnonlocal}, that is
\begin{equation*}
	\tau h''_j+\gamma_{{}_{F,g}} h'_j=\frac{\e}{c_{{}_F}}\left(\alpha^{j+1}-\alpha^j+\frac{(-1)^{j+1}}{N+1}\sum_{i=1}^{N+1}(-1)^i(\alpha^{i+1}-\alpha^i)\right),
\end{equation*}
for $j=1,\dots,N+1$. 
By (formally) taking $\tau=0$ and $\gamma_{{}_{F,g}}=1$, 
one obtains the ODEs describing the layer dynamics in the case of the mass conserving Allen--Cahn equation \eqref{eq:maco-AC}.

\end{document}